\documentclass[11pt]{article}

 \usepackage{amsmath,amssymb,amscd,amsthm,esint}
 
\usepackage{graphics,amsmath,amssymb,amsthm,mathrsfs}

\usepackage{graphics,amsmath,amssymb,amsthm,mathrsfs,amsfonts}

\usepackage{sidecap}
\usepackage{float}
\usepackage{extarrows}
\usepackage{booktabs}
\usepackage{verbatim}
\usepackage{hyperref}
\usepackage[usenames,dvipsnames]{xcolor}

\newtheorem{thm}{Theorem}[section]
\newtheorem{lemma}[thm]{Lemma}

\theoremstyle{definition}
\newtheorem{remark}[thm]{Remark}

\def\XXint#1#2#3{{\setbox0=\hbox{$#1{#2#3}{\int}$}
         \vcenter{\hbox{$#2#3$}}\kern-.5\wd0}}

\def\R{\mathbb{R}}

\def\e{\varepsilon}

\def\tQ{\widetilde{Q}}

\def\Z{\mathbb{Z}}
\def\e{\varepsilon}

\numberwithin{equation}{section}

\begin{document}

\title{ Dirichlet Problems  in Perforated Domains}

\author{
Robert Righi\thanks{Supported in part by NSF grants DMS-1856235 and  DMS-2153585.} 
\qquad
Zhongwei Shen\thanks{Supported in part by NSF grants DMS-1856235 and DMS-2153585.}  }
\date{}
\maketitle

\begin{abstract}

 In this paper we establish $W^{1,p}$ estimates for solutions $u_\e$ to Laplace's equation 
 with the Dirichlet condition in a bounded and perforated, not necessarily periodically,   $C^1$  domain $\Omega_{\e, \eta}$
  in $\mathbb{R}^d$. The bounding constants depend explicitly on two small parameters $\e$ and $\eta$, where $\e$ 
  represents the scale of the minimal distance between holes, and $\eta$ denotes the ratio between the size of the holes and $\e$. 
  The proof relies on a large-scale $L^p$ estimate for $\nabla u_\e$, whose proof is divided into two parts.
  In the first part, we show that 
  as  $\e, \eta $ approach zero,  harmonic functions in $\Omega_{\e, \eta}$ may be approximated by solutions of
    an intermediate problem for a Schr\"odinger operator in $\Omega$.
     In the second part,  a real-variable method is employed to establish the large-scale $L^p$ estimate for $\nabla u_\e$ by using
     the approximation at scales above $\e$. The results are sharp except in the case $d\ge 3$ and $p=d$ or $d^\prime$.
     
\medskip


\medskip

\noindent {\it MR (2020) Subject Classification}: 35Q35;  35B27; 76D07.

\end{abstract}


\section{Introduction}\label{section-1}

In this paper we continue the study of the Dirichlet problem for Laplace's equation, 
\begin{equation}\label{DP-0}
\left\{
\aligned
-\Delta u_\e & = F + \text{\rm div}(f) & \quad & \text{ in } \Omega_{\e, \eta},\\
u_\e & = 0 & \quad & \text{ on } \partial \Omega_{\e, \eta},
\endaligned
\right.
\end{equation}
where $\Omega_{\e, \eta} $ is a domain perforated, not necessarily periodically, with a large number of tiny holes.
Given $F \in L^p(\Omega_{\e, \eta}) $ and $ f\in L^p(\Omega_{\e, \eta}; \R^d)$,
it is well known that the Dirichlet problem \eqref{DP-0} possesses a unique solution 
$u_\e=u_{\e, \eta}$ in $W_0^{1, p}(\Omega_{\e, \eta} )$,  if $1< p< \infty$ and $\Omega_{\e, \eta} $ is a bounded $C^1$ domain in $\R^d$, 
$d\ge 2$.
Let $A_p(\Omega_{\e, \eta} )$ and $B_p(\Omega_{\e, \eta} )$ denote the smallest constants for which the $W^{1, p}$ estimate,
\begin{equation}\label{W1p-0}
\| \nabla u_\e \|_{L^p(\Omega_{\e, \eta} )}
\le A_p (\Omega_{\e, \eta} )  \| f \|_{L^p(\Omega_{\e, \eta} )} + B_p(\Omega_{\e, \eta} )\|F\|_{L^p(\Omega_{\e, \eta} )},
\end{equation}
holds. We are interested in  bounds of $A_p(\Omega_{\e, \eta})$ and $B_p(\Omega_{\e, \eta})$ that 
exhibit explicit and sharp dependence  on the sizes of the holes as well as on the distances between them.

Let $Q(x, r)$ denote the open cube centered at $x$ with side length $r$.
To describe the perforated domain $\Omega_{\e, \eta}$, 
let  $\{Y_z^s: z\in \Z^d\}$  be a family of  bounded domains with connected and uniform $C^1$ boundaries such that  
\begin{equation}\label{c-0}
B(0, c_0)\subset Y_z^s \subset B(0, 1/8)
\end{equation}
for some $c_0>0$.
Let $\{ x_z: z\in \Z^d\}$ be a sequence of points in $ Q(0, 1/2)$ and
\begin{equation}\label{T}
T_{z, \eta} =z+ x_z + \eta \overline{ Y_z^s},
\end{equation}
where $\eta\in (0, 1/2)$. 
For a domain $\Omega$ in $\R^d$ and $0< \e\le 1$, define
\begin{equation}\label{O-e}
\Omega_{\e, \eta}  =\Omega\setminus \bigcup_z \e {T_{z, \eta} },
\end{equation}
where the union is taken over those $z$'s in $\Z^d$ for which  $\e (z+Y) \subset \Omega$, where $Y=Q(0, 1)$.
Thus, the perforated domain $\Omega_{\e, \eta}$ is obtained from $\Omega$ by removing a hole
$\e { T_{z, \eta} }$, centered at $\e(z+x_z)$ and of size $\e \eta$, from each cell  $\e (z+Y)$ of size $\e$
and contained in $\Omega$.
Roughly speaking, the parameter $\e$ represents the scale of the minimal distances between holes, while 
the parameter $\eta$ represents the scale of the ratios between the sizes of the holes and $\e$.
We point out that the holes need not to be  identical, nor they are placed periodically, 
unless the sequences  $\{x_z\}$, and $\{ Y_z^s\}$ are independent of $z$.

To state the main results for $A_p(\Omega_{\e, \eta})$, 
we note that $A_2(\Omega_{\e, \eta})=1$, and that $A_p (\Omega_{\e, \eta})=A_{p^\prime}(\Omega_{\e, \eta} )$, where
$p^\prime=\frac{p}{p-1}$, by duality.
As a result, it suffices to consider the case $2< p< \infty$.
Let
\begin{equation}\label{sigma}
\sigma_\e=\sigma_{\e, \eta}
=\left\{
\aligned
& \e \eta^{1-\frac{d}{2}} & \quad & \text{ if } d\ge 3,\\
& \e |\ln \eta|^{\frac12} & \quad & \text{ if } d=2.
\endaligned
\right.
\end{equation}
The asymptotic behaviors of $A_p(\Omega_{\e, \eta} )$ and $B_p(\Omega_{\e, \eta})$, as $\e, \eta \to 0$,
depend on $\sigma_\e$.
 Our first theorem treats the case of relatively large holes, where $\sigma_\e \le 1$, while our second theorem handles the case of relatively small holes, 
 where $\sigma_\e \geq 1$.

\begin{thm}\label{main-thm-1}
Suppose that $0< \sigma_\e \le 1$ and $2< p< \infty$.
Let $\Omega$ be a bounded $C^1$ domain in $\R^d$ and $\Omega_{\e, \eta}$ be given by \eqref{O-e}.
Then
\begin{equation}\label{W1p-01}
A_p (\Omega_{\e, \eta} ) \le 
\left\{
\aligned
&C \eta^{-d|\frac12-\frac{1}{p}|}  & \quad & \text{ if }  d\ge 3,\\
& C \eta^{-2|\frac{1}{2}-\frac{1}{p}|} |\ln \eta|^{-\frac12}
& \quad &\text{ if }  d=2,
\endaligned
\right.
\end{equation}
where $C$ depends only on $d$, $p$, $\Omega$, and $\{Y_z^s\}$.
\end{thm}

\begin{thm}\label{main-thm-2}
Suppose  that $\sigma_\e \ge 1$ and $2< p< \infty$.
Let $\Omega$ be a bounded $C^1$ domain in $\R^d$ and $\Omega_{\e, \eta}$ be given by \eqref{O-e}.
 Then 
\begin{equation}\label{W1p-02}
A_p(\Omega_{\e, \eta}) \le 
\left\{
\aligned
& C (1+ \e^{-1} \eta^{\frac{d}{p}-1}) & \quad & \text{ if }  d\ge 3 \text{ and } 2<p<d,\\
& C (\e^{-1} + |\ln \eta|^{1-\frac{1}{d}}) & \quad & \text{ if }   d\ge 3 \text{ and } p=d,\\
& C \e^{-1} \eta^{\frac{d}{p}-1} & \quad & \text{ if }   d\ge 3 \text{ and } d< p<\infty,\\
& C \e^{-1} \eta^{\frac{2}{p}-1} |\ln \eta|^{-1}
& \quad &\text{ if }  d=2   \text{ and } 2<p< \infty,
\endaligned
\right.
\end{equation}
where $C$ depends only on $d$, $p$, $\Omega$, and $\{Y_z^s\}$.
\end{thm}
We point out that the upper bounds for $A_p(\Omega_{\e, \eta} )$ in \eqref{W1p-01} are sharp.
 The upper bounds of $A_p(\Omega_{\e, \eta} )$ in \eqref{W1p-02}
are also sharp for $d=2$ as well as for $d\ge 3$ and $p\neq d$.
Whether the upper bound is  sharp for the remaining case, where $d\ge 3$ and $p=d$ (or $d^\prime$),  is not known.
Indeed, if $\Omega_{\e, \eta}$ is a periodically perforated domain, given by \eqref{O-e}  with the sequences
 $\{ x_z\}$ and $\{Y_z^s\}$ independent of $z$, 
we will show that 
\begin{equation}\label{W1p-01s}
A_p(\Omega_{\e, \eta} )   \ge 
\left\{
\aligned
&c \, \eta^{-d|\frac12-\frac{1}{p}|}  & \quad & \text{ if }   d\ge 3,\\
& c\,  \eta^{-2|\frac{1}{2}-\frac{1}{p}|} |\ln \eta|^{-\frac12}
& \quad &\text{ if }   d=2,
\endaligned
\right.
\end{equation}
in the case $\sigma_\e  \le 1$, 
and that 
\begin{equation}\label{W1p-02s}
A_p(\Omega_{\e, \eta} ) \geq 
\left\{
\aligned
& c (1+ \e^{-1} \eta^{\frac{d}{p}-1}) & \quad & \text{ if }   d\ge 3 \text{ and } 2<p<d,\\
& c \e^{-1} \eta^{\frac{d}{p}-1} & \quad & \text{ if }  d\ge 3 \text{ and } d\le  p<\infty,\\
& c \e^{-1} \eta^{\frac{2}{p}-1} |\ln \eta|^{-1}
& \quad &\text{ if } d=2 \text{ and } 2< p < \infty,
\endaligned
\right.
\end{equation}
in the case $\sigma_\e \ge  1$.
The constants $c>0$ in \eqref{W1p-01s}-\eqref{W1p-02s}  are independent of $\e$ and $\eta$.
See Theorem \ref{thm-A-low-1}. 
Note that in the case $d\ge 3$ and  $p=d$,
our upper and lower bounds differ by a term $C |\ln \eta|^{1-\frac{1}{d}}$.

The main results for $B_p(\Omega_{\e, \eta})$ are given in Theorems \ref{addin-thm}, \ref{main-thm-3} and \ref{main-thm-4}.
The question of their sharpness is addressed in Section \ref{section-Sh}.

Motivated by the homogenization theory of boundary value problems for elliptic equations in perforated domains \cite{lions1980asymptotic, cioranescu2018strange, allaire1991homogenization,allaire1991homogenization2, jing2020unified, MR4374607},
the study of the uniform $W^{1, p}$ estimate \eqref{W1p-0} in $\Omega_{\e, \eta}$
 was initiated by N. Masmoudi  \cite{MR2111721} in the case $\eta=1$.
 The general case $0< \eta<1$ was first studied  by 
  the second author of the present paper in \cite{shen2023uniform}. 
 Both the bounded domain $\Omega_{\e, \eta} $ and the unbounded perforated domain, 
 \begin{equation}\label{omega}
 \omega_{\e, \eta} =\mathbb{R}^d \setminus \bigcup_{z\in \mathbb{Z}^d} \e ( z+ \eta \overline{Y_z^s}),
 \end{equation}
  were considered in  \cite{shen2023uniform}, where  $\Omega_{\e, \eta}$ and $\omega_{\e, \eta}$ need not to be periodically perforated.
  The upper bounds  obtained  in \cite{shen2023uniform} for $A_p$ are not sharp, although  they are only off by an arbitrary small power of $\eta$.
  In \cite{MR4630318} J. Wallace and  the second author  studied  the case of the unbounded domain $\omega_{\e, \eta}$.
  Under the additional assumption that  the holes $\{ Y^s_z \}$  are identical and thus $\omega_{\e, \eta}$ is periodic,
  they were able to obtain the estimate \eqref{W1p-0} with sharp constants $A_p$ and $B_p$.
  The argument  in \cite{MR4630318} uses a large-scale Lipschitz estimate for harmonic functions $u_\e$ in perforated domains,
  whose proof relies on  the observation that the difference  $u_\e(x+\e e_j) - u_\e(x)$ is also harmonic.
  It is not clear how to extend this proof to the non-periodic case  and to the case of  bounded domains.
  
  In this paper we introduce a new approach to the study of $W^{1, p}$ estimates in perforated domains, which allows us to handle the case of non-periodic holes as well as
  boundary estimates for bounded perforated domains.   
To show Theorems \ref{main-thm-1} and  \ref{main-thm-2},
 we start with a localization argument as in \cite{MR4630318} and reduce the $L^p$ estimates of $\nabla u_\e $ to an $L^p$ estimate of the operator $S_\e$ defined by
\begin{equation}\label{T-op}
S_\e (F, f)=\left(\fint_{Q(x, 2\e) } |\nabla u_\e|^2 \right)^{1/2}, 
\end{equation}
where $u_\e$ is the solution of \eqref{DP-0},  extended to $\R^d$ by zero. It  is not hard to see that
$$
\| S_\e(F,f)\|_{L^2(\R^d)} = \|\nabla u_\e \|_{L^2(\Omega_{\e, \eta} )}.
$$
Thus, by energy estimates,
\begin{equation}\label{est-E}
\| S_\e ( F, f) \|_{L^2(\R^d)}
\le \| f \|_{L^2(\Omega_{\e, \eta} )} + C \min (\sigma_\e, 1) \| F \|_{L^2(\Omega_{\e, \eta})}.
\end{equation}

\begin{thm}\label{thm-S0}
Let $\Omega$ be a bounded $C^1$ domain in $\R^d$ and $\Omega_{\e, \eta}$ be given by \eqref{O-e}.
Then
\begin{equation}\label{S-0}
\| S_\e (F, f)\|_{L^p(\R^d)}
\le C \left\{  \| f \|_{L^p(\Omega_{\e, \eta})}
+ \min (\sigma_\e, 1) \| F \|_{L^p(\Omega_{\e, \eta}) } \right\}
\end{equation}
 for $2<p< \infty$, 
 where $C$ depends only on $d$, $p$, $\{Y_z^s\}$ and $\Omega$.
\end{thm}

By a localization argument as well as $L^p$ estimates for $u_\e$ in \cite{shen2023uniform},
Theorems \ref{main-thm-1} and \ref{main-thm-2} follow  from  Theorem \ref{thm-S0}.
Since $S_\e$ is defined by averaging $\nabla u_\e$ over a cell of size $2\e$, 
the estimate \eqref{S-0} should be regarded as a large-scale  $W^{1, p}$ estimate for $u_\e$.
Theorem \ref{thm-S0} shows that  such  average of $\nabla u_\e$  behaves much better than $\nabla u_\e$ in $L^p$ spaces for $p\neq 2$.
Indeed, as $\e, \eta \to 0$, $A_p(\Omega_{\e, \eta} ) \to \infty$ for $p\neq 2$ and $\sigma_\e\le 1$,
while the operator norm $\| S_\e (0,  \cdot ) \|_{L^p \to L^p}$ remains bounded.

To prove Theorem \ref{thm-S0}, we use  a real-variable method  from \cite{shenbounds}, which reduces the problem to a reverse 
H\"older inequality  for harmonic functions $u_\e$  in perforated domains.
The proof of the reverse H\"older inequality again relies on the same real-variable argument. 
As such, for each subdomain  $D$ of size greater than $\e$, one needs to approximate $\nabla u_\e$ by a 
function that behaves well in the $L^p$ norm.
To this end, for each hole in the domain, we introduce an associated nonnegative potential,
supported in a neighborhood of the hole. 
Let $V_\e$ denote the sum of these potentials.
A key observation of this paper  is that a harmonic function $u_\e$ in a perforated domain $D_{\e, \eta}$ is well approximated by 
 $\chi_{\e, \eta} v_\e$, where $\chi_{\e, \eta}$ is a corrector for $D_{\e, \eta}$ with $\chi_{\e, \eta}=1$ on $\partial D$, 
 and $v_\e$ is the solution of  an intermediate equation,  
\begin{equation}\label{IP-0}
(-\Delta + \sigma_\e^{-2} V_{\e}) v_{\e} =0
\end{equation} 
 in the non-perforated domain $D$ with boundary data $v_\e =u_\e$ on $ \partial D$.
We note that even though the holes may not be periodically placed, the potential $V_{\e}$ behaves like a constant.
As a result,  $\nabla v_{\e}$ satisfies the reverse H\"older inequality for any $p>2$.

The rest of the paper is organized as follows.
 In Section \ref{section-R} we  give the proof of Theorems \ref{main-thm-1} and \ref{main-thm-2}, using Theorem \ref{thm-S0}.
 The estimates of bounding constants $B_p(\Omega_{\e, \eta})$ in \eqref{W1p-0} are given in Section \ref{section-B}.
 Sections \ref{section-C1} - \ref{section-L} are devoted to the proof of Theorem \ref{thm-S0}.
 Specifically,  the construction and estimates of the correctors $\chi_{\e, \eta}$ are given in Sections \ref{section-C1} and \ref{section-C2}.
 The estimates for solutions of the intermediate problem \eqref{IP-0} are given in Section \ref{section-I}.
 A   convergence rate for $ \nabla (u_\e - \chi_{\e, \eta} v_\e ) $ in $L^2$  is established in Section \ref{section-Con}, and 
 the  proof of Theorem \ref{thm-S0} is given in Section \ref{section-L}.
 Finally, we address the question of sharpness of Theorems \ref{main-thm-1} and \ref{main-thm-2} in Section \ref{section-Sh}.


\section{Reduction to the large-scale estimates}\label{section-R}

Recall  $Q(x, r)$ denotes the open cube centered at $x$ with size length $r$.
Let $\{Y_z^s: z\in \mathbb{Z}^d\}$ be a family of bounded domains with connected and uniform $C^1$ boundaries such that 
$B(0, c_0) \subset Y_z^s\subset B(0, 1/8)$.
Let $\{ x_z: z\in \mathbb{Z}^d\}$ be a sequence of points in $Q(0, 1/2)$.
For $z\in \mathbb{Z}^d$, we define  $Q_z =z+ Q(0, 1)$ and $T_{z, \eta} =z+x_z + \eta \overline{Y_z^s}$, where
$\eta \in (0, 1/2)$.
Note that $B(z+x_z, c_0 \eta) \subset T_{z, \eta} \subset Q_z$ and dist$(T_{z, \eta}, \partial Q_z) \ge c>0$.

 The following lemma  will be useful to us.
 Its proof for $p =2$ is well known \cite{allaire1991homogenization}. 
 The proof for $p\neq 2$  is similar and can be found in \cite[Lemma 2.1]{shen2023uniform}. 

\begin{lemma}\label{lemma-P0}
 Let $d \geq 2$ and $1 \le  p < \infty$. Suppose that $u \in W^{1,p}(Q(0, \e) )$ and $u =0$ on $B(x_0,\e \eta)$ for some $x_0\in  Q(0, \e /2)$ and $0 <\eta < 1/4$. Then 
 \begin{equation}\label{PI}
     \int_{Q(0, \e) } |u|^p\,  dx \leq C \int_{Q(0, \e)}  |\nabla u|^p \, dx \cdot \begin{cases}
         \begin{alignedat}{2}
             &\e^p \eta^{p-d} \qquad &&\text{if} \ 1 \le  p < d, \\ 
             &\e^p |\ln \eta|^{d-1 } \qquad &&\text{if} \ p =d,  \\ 
             &\e^p  \qquad &&\text{if} \ d < p < \infty,
         \end{alignedat}
     \end{cases}
 \end{equation}
 where  $C$ depends on $d$ and $p$.
\end{lemma}

\begin{lemma}\label{Lemma-P}
Let $\Omega$ be a bounded Lipschitz domain in $\mathbb{R}^d$ 
 and  $\Omega_{\e, \eta}$ be  given by \eqref{O-e}. Then
\begin{equation}\label{P00}
\| u \|_{L^2 (\Omega_{\e, \eta})}
\le C \min (\sigma_\e, 1) \|\nabla  u \|_{L^2(\Omega_{\e, \eta})}
\end{equation}
for any $u \in H_0^1(\Omega_{\e, \eta})$,
where $\sigma_\e$ is given by \eqref{sigma} and $C$ is independent of $\e$ and $\eta$.
\end{lemma}

\begin{proof}
 It follows from  Lemma \ref{lemma-P0} that if $u \in H^1(\Omega_{\e, \eta})$ and $u= 0 $ on $  \partial\Omega_{\e, \eta} \setminus \partial\Omega$, then 
\begin{equation} \label{P-01}
 \| u \|_{L^2(\Omega_{\e, \eta})} \le C \sigma _\e \|\nabla u \|_{L^2(\Omega_{\e, \eta})}.
 \end{equation}
 This,  together with the Poincar\'e inequality $\| u \|_{L^2(\Omega)} \le C \|\nabla u \|_{L^2(\Omega)} $ for
 $u \in H_0^1(\Omega)$, yields
 \eqref{P00}
for $u \in H_0^1(\Omega_{\e, \eta})$.
\end{proof}

For $F\in L^2(\Omega)$ and $f\in L^2(\Omega; \R^d)$, let $u_\e\in H_0^1(\Omega_{\e, \eta})$ be the  solution of the Dirichlet problem 
\eqref{DP-0}.
Using \eqref{P00}, it is not hard to show that
\begin{equation}\label{P01}
\| \nabla u_\e \|_{L^2(\Omega)}
\le \| f\|_{L^2(\Omega)}
+ C \min (\sigma_\e, 1) \| F \|_{L^2(\Omega)},
\end{equation}
where we have extended $u_\e$ to $\Omega$ by zero.
The following theorem gives the $L^p$ estimate of $u_\e$ for $p> 2$.

\begin{thm}\label{lemma-P1}
Let $F \in L^p(\Omega)$ and $f\in L^p(\Omega; \R^d)$ for some $p\ge 2$.
Let $u_\e$ be the solution of \eqref{DP-0}, where $\Omega$ is a bounded Lipschitz domain. Then
\begin{equation}\label{P1-0}
\| u_\e \|_{L^p(\Omega)}
\le C \min (\sigma_\e, 1) \| f \|_{L^p(\Omega)} + C \min (\sigma_\e^2, 1) \| F \|_{L^p(\Omega)},
\end{equation}
where $C$ depends only on $d$, $p$, $c_0$ and $\Omega$.
\end{thm}

\begin{proof}
The case $p=2$ follows readily from the energy estimate \eqref{P01} and \eqref{P00}.
The case $p>2$ was proved in \cite{shen2023uniform}.
\end{proof}

To establish the $L^p$ estimates for $\nabla u_\e$, 
we consider the local $W^{1,p} $ estimates for solutions of 

 \begin{equation}\label{cellprob-4}
 \left\{
   \aligned
       -\Delta u_\e  & = F + \text{\rm div} (f) & \quad &  \text{ in }   \e (\tQ_z\setminus T_{z, \eta}),   \\ 
       u_\e  & = 0  & \quad &  \text{ in }  \e T_{z, \eta}, 
   \endaligned
   \right.
 \end{equation}
where $z\in \mathbb{Z}^d$ and  $\tQ_z=   Q(z, 17/16)$. Let 
  \begin{equation}\label{PhiR-4}
        \Phi_p(R) = \begin{cases}
            1 & \text{if} \ d \geq 3 \ \text{and} \ 2 <p < d, \\
            (\ln R)^{1 - \frac{1}{d}} & \text{if} \ d \geq 3 \ \text{and} \ p = d, \\ 
            R^{1-\frac{d}{p}} & \text{if} \ d \geq 3 \ \text{and} \ d<p<\infty, \\ 
            R^{1-\frac{2}{p}}(\ln R)^{-1} & \text{if} \ d=2 \ \text{and} \ 2 < p < \infty,
        \end{cases}   
    \end{equation}
    for $R>2$.

\begin{lemma}\label{Main4-4}
Let $2 < p < \infty$. 
 Let $u_\e$ be a solution of \eqref{cellprob-4} with $F \in L^p(\e \tQ_z) $ and $f \in L^p (\e \tQ_z ; \mathbb{R}^d)$. 
 Then for $d\geq 3$,
\begin{equation}\label{mainloces-4}
\begin{split}
    \left(\fint_{\e Q_z} |\nabla u_\e|^p \right)^{1/p}
    \leq C|\alpha| \e^{-1 } \eta^{\frac{d}{p}-1} + C \Phi_p(\eta^{-1} ) 
    \left( \fint_{\e \tQ_z} ( |\e F| + |f|)^p \right)^{1/p} 
       \\ + C \e^{  -1}\Phi_p(\eta^{-1})
       \left(\fint_{\e \tQ_z\backslash Q (y_z, \e/3)} |u_\e - \alpha |^2 dx  \right)^{1/2}, 
    \end{split}
\end{equation}
and for $d=2$
\begin{equation}\label{mainlocesd2-4}
\begin{split}
       \left(\fint_{\e Q_z} |\nabla u_\e|^p \right)^{1/p}
     \leq C|\alpha| \e^{ -1 } \Phi_p(\eta^{-1})
    + C \Phi_p(\eta^{-1} ) \left( \fint_{\e \tQ_z} ( |\e F| + |f|)^p  \right)^{1/p} 
       \\ + C \e^{  -1}\Phi_p(\eta^{-1})\left(\fint_{\e\tQ_z\backslash Q (y_z, \e/3)} |u_\e - \alpha |^2 dx  \right)^{1/2}, 
    \end{split}
\end{equation}
where $\alpha \in \R$, $y_z=z+x_z$ and $C$ depends only on $d$,  $p$ and $\{ Y_z^s\}$.
\end{lemma}

\begin{proof}
The case where  $\e=1$ and $x_z=0$ was proved in \cite[Theorem 6.1]{MR4630318}, using a result in 
\cite{amrouche1997dirichlet}. 
 The same argument, together with the correctors constructed in this paper, 
 also gives \eqref{mainloces-4}-\eqref{mainlocesd2-4} for $\e=1$ and $x_z \in Q(0, 1/2)$ with constants
 independent of $x_z$.
 We point out that there is a minor flaw in the proof of Theorem 6.1 in \cite{MR4630318}.
In the case $d=2$, 
 the corrector defined in \cite[(6.6)]{MR4630318} does not satisfy the condition in \cite[(6.11)]{MR4630318}.
 The mistake can be fixed easily by using the corrector in Section \ref{section-C2}.
 
The general case $0< \e <1$ follows readily from the case $\e=1$ by dilation.
Note that the parameter $\eta$ remains invariant under dilation.
\end{proof}

\begin{remark}\label{re-P1}

Let $u_\e$ be the  solution of \eqref{DP-0} with $F\in L^p(\Omega)$ and $f\in L^p(\Omega; \R^d)$ for some $p>2$.
Let 
\begin{equation}\label{O-ee}
\Omega^\prime_{\e, \eta}
=\bigcup_z \e (Q_z \setminus T_{z, \eta}),
\end{equation}
where the union is taken over those $z$'s in $\mathbb{Z}^d$  for which $\e Q_z \subset \Omega$.
It follows from \eqref{mainloces-4}-\eqref{mainlocesd2-4} with $\alpha=0$ and a covering argument  that 
\begin{equation}\label{P-2a}
\int_{\Omega^{\prime\prime}_{\e, \eta}} |\nabla u_\e|^p 
\le C |\Phi_p (\eta^{-1})|^p 
\int_\Omega (\e|F|+  |f|)^p
+ C \e^{-p} |\Phi_p (\eta^{-1})|^p \int_{\Omega} |u_\e |^p,
\end{equation}
where $\Omega_{\e, \eta}^{\prime\prime}= \{ x \in \Omega^\prime_{\e, \eta} : \text{\rm dist}(x, \partial \Omega)\ge c\e \}$.
For the region near $\partial\Omega$ and away from the holes, we use the boundary $W^{1, p}$
estimates for Laplace's equation in $C^1$ domains \cite{MR1331981}  to obtain
\begin{equation}
\int_{B(x, c\e)} |\nabla u_\e|^p
\le C \e^{-p} \int_{B(x, 2c\e)}  |u_\e|^p 
+C \int_{B(x, 2c\e)} (\e |F|+ |f|)^p,
\end{equation}
where $x\in \Omega$ and $B(x, 2c\e)\cap \Omega =B(x, 2c\e) \cap \Omega_{\e, \eta}$.
By a covering argument we obtain 
\begin{equation}\label{P-2b}
\int_{\Omega_{\e, \eta}\setminus \Omega_{\e, \eta}^{\prime\prime} } |\nabla u_\e|^p
\le C \int_{\Omega} ( \e |F|+ |f|)^p
+  C \e^{-p} \int_{\Omega_{\e, \eta} } |u_\e|^p.
\end{equation}
This, together with \eqref{P-2a}, gives
\begin{equation*}
\aligned
\int_\Omega |\nabla u_\e|^p
 & \le C |\Phi_p (\eta^{-1})|^p 
\int_\Omega ( \e |F| + |f|)^p
+ C \e^{-p} |\Phi_p (\eta^{-1})|^p \int_{\Omega} |u_\e |^p\\
&\le C  |\Phi_p(\eta^{-1}) |^p 
\left( 1 + \e^{-1} \min (\sigma_\e, 1)\right)^p \int_\Omega | f|^p\\
 &\qquad \qquad  + C   |\Phi_p(\eta^{-1}) |^p 
\left( 1 + \e^{-1} \min (\sigma_\e, 1)\right)^{2p}  \int_\Omega | \e F |^p, 
\endaligned
\end{equation*}
where we have used \eqref{P1-0} for the last inequality.
As a result, we have proved that 
\begin{equation}\label{P-2c}
\aligned
\| \nabla u_\e\|_{L^p(\Omega)}
 & \le C \Phi_p(\eta^{-1}) ( 1+ \e^{-1} \min (\sigma_\e, 1)) \| f \|_{L^p(\Omega)}\\
 & \quad \qquad
 + C \Phi_p (\eta^{-1}) \left( 1+\e^{-1} \min (\sigma_\e, 1) \right)^2
\| \e F \|_{L^p(\Omega)}
\endaligned
\end{equation}
for $p>2$ and $d\ge 2$.
Suppose $F=0$.
In the case $0< \sigma_\e \le 1$, $d\ge 3$ and $p>d$, it follows from \eqref{P-2c} that 
\begin{equation}\label{P-2d}
\|\nabla u_\e \|_{L^p(\Omega)}
\le C  \eta^{\frac{d}{p} -\frac{d}{2}}  \| f \|_{L^p(\Omega)}.
\end{equation}
Since  \eqref{P-2d} also holds for $p=2$, by Riesz Thorin  Interpolation Theorem,
\eqref{P-2d} holds for $2< p<\infty$.
This gives  the proof of Theorem \ref{addin-thm} for the case $d\ge 3$.
\end{remark}

\begin{remark}\label{re-P2}

Suppose $\sigma_\e \ge 1$. It follows from \eqref{P-2c} that if $F=0$, 
\begin{equation}\label{P-3a}
\|\nabla u_\e \|_{L^p(\Omega)}
\le C  \e^{-1} \Phi_p (\eta^{-1})   \| f \|_{L^p(\Omega)}.
\end{equation}
It follows that 
\begin{equation}\label{P-3b}
\|\nabla u_\e \|_{L^p(\Omega)}
\le 
\left\{
\aligned
& C \e^{-1} \eta^{\frac{d}{p}-1}  \| f\|_{L^p(\Omega)} & \quad & \text{ if } d\ge 3 \text{ and } p>d,\\
& C \e^{-1} \eta^{\frac{2}{p} -1} |\ln \eta|^{-1}  \| f \|_{L^p(\Omega)} & \quad & \text{ if } d=2 \text{ and } 2< p< \infty.
\endaligned
\right.
\end{equation}
This gives the estimates in Theorem \ref{main-thm-2} for the case $d\ge 2$ and $d< p< \infty$.
\end{remark}

To deal with the remaining cases in Theorem \ref{main-thm-1} and \ref{main-thm-2}, we let
$$
\alpha = \fint_{\e \tQ_z\setminus Q(y_z, \e/3)} u_\e
$$
in \eqref{mainloces-4}-\eqref{mainlocesd2-4}.
By Poincar\'e's inequality we obtain 
\begin{equation}\label{P-4a}
\aligned
\left(\fint_{\e Q_z} |\nabla u_\e|^p \right)^{1/p}
\le C \e^{-1} \eta^{\frac{d}{p}-1} \fint_{\e \tQ_z} |u_\e|
 & +  C \Phi_p (\eta^{-1})
\left(\fint_{\e \tQ_z} (|\e F| + |f|)^p \right)^{1/p}\\
& + C \Phi_p (\eta^{-1}) \left(\fint_{\e \tQ_z} |\nabla u_\e|^2 \right)^{1/2}
\endaligned
\end{equation}
for $d\ge 3$, and
\begin{equation}\label{P-4b}
\aligned
\left(\fint_{\e Q_z} |\nabla u_\e|^p \right)^{1/p}
\le C \e^{-1} \Phi_p(\eta^{-1} ) \fint_{\e \tQ_z} |u_\e|
 & + C \Phi_p (\eta^{-1})
\left(\fint_{\e \tQ_z} (|\e F| + |f|)^p \right)^{1/p}\\
& + C \Phi_p (\eta^{-1}) \left(\fint_{\e \tQ_z} |\nabla u_\e|^2 \right)^{1/2}
\endaligned
\end{equation}
for $d=2$.
Define
\begin{equation*}
S_\e ( F, f )(x) =\left(\fint_{ Q(x, 2\e )} |\nabla u_\e|^2 \right)^{1/2}.
\end{equation*}
It follows from \eqref{P-4a}-\eqref{P-4b} by  a covering argument  that 
\begin{equation}\label{P-4c}
\aligned
\int_{\Omega_{\e, \eta}^{\prime\prime}}
|\nabla u_\e|^p
\le C \e^{-p} \eta^{d-p}  \int_{\Omega} |u_\e|^p
 & + C |\Phi_p (\eta^{-1}) |^p 
\int_{\Omega} (| \e F| +|f|)^p\\
& + C| \Phi_p(\eta^{-1})|^p \int_\Omega |S_\e (F, f)|^p
\endaligned
\end{equation}
for $d\ge 3$, and that 
\begin{equation}\label{P-4d}
\aligned
\int_{\Omega_{\e, \eta}^{\prime\prime}}
|\nabla u_\e|^p
\le C \e^{-p} |\Phi_p(\eta^{-1}) |^p \int_{\Omega} |u_\e|^p
 & + C |\Phi_p (\eta^{-1}) |^p 
\int_{\Omega} (| \e F| +|f|)^p\\
& + C \Phi_p(\eta^{-1})|^p \int_\Omega |S_\e (F, f)|^p
\endaligned
\end{equation}
for $d=2$, where $\Omega^{\prime\prime} _{\e, \eta}$ is  the same as in  \eqref{P-2a}.
For the region near $\partial \Omega$ and away from the holes, we use the  local $W^{1, p}$ estimate in $C^1$ domains,
\begin{equation}\label{local-a}
\left(\fint_{Q(x, c\e)} |\nabla u_\e|^p \right)^{1/p}
\le C \left(\fint_{Q(x, 2\e)} |\nabla u_\e|^2 \right)^{1/2}
+ C \left(\fint_{Q(x, 2\e)} (|\e F | + |f|)^p\right)^{1p},
\end{equation}
where $x\in \Omega$,  dist$(x, \partial\Omega) \le C \e$ and $B(x, 2c \e)\cap \Omega=B(x, 2c\e)\cap \Omega_{\e, \eta}$.
By a covering argument, this leads to
\begin{equation}\label{B-e0}
\int_{\Omega_{\e, \eta} \setminus \Omega_{\e, \eta}^{\prime\prime}}
|\nabla u_\e|^p 
\le C \int_\Omega |S_\e (F, f)|^p
+ C \int_\Omega (|\e F| +|f|)^p.
\end{equation}

\begin{lemma}
Let $\Omega$ be a bounded $C^1$ domain in $\R^d$.
Let $u_\e\in H_0^1(\Omega_{\e, \eta})$ be the  solution of  \eqref{DP-0} with $F\in L^p(\Omega)$ and $f\in L^p(\Omega; \R^d)$ for some $p>2$.
Then
\begin{equation}\label{P-5a}
\aligned
\|\nabla u_\e \|_{L^p(\Omega)}
 & \le C \e^{-1} \eta^{\frac{d}{p}-1} \| u_\e \|_{L^p(\Omega)}\\
& \qquad+ C \Phi_p (\eta^{-1} ) \left\{
\| \e |F| + |f| \|_{L^p(\Omega)}
+ \| S_\e (F, f) \|_{L^p(\Omega)} \right\}
\endaligned
\end{equation}
for $d\ge 3$, and
\begin{equation}\label{P-5b} 
\aligned
\|\nabla u_\e \|_{L^p(\Omega)}
 & \le C \e^{-1}  \Phi_p (\eta^{-1} ) \| u_\e \|_{L^p(\Omega)}\\
 &\qquad
 + C \Phi_p (\eta^{-1})  \left\{
\| \e |F| + |f| \|_{L^p(\Omega)}
+ \| S_\e (F, f) \|_{L^p(\Omega)} \right\}
\endaligned
\end{equation}
for $d=2$, where $C$ depends on $d$, $p$, $\{Y_z^s\}$ and $\Omega$.
\end{lemma}

\begin{proof}
This follows from \eqref{P-4c}- \eqref{P-4d} and \eqref{B-e0}.
\end{proof}

Using  Theorem \ref{thm-S0}, we now give the proof of Theorems  \ref{main-thm-1} and \ref{main-thm-2}.

\begin{proof}[Proof of Theorem \ref{main-thm-1}]

Suppose $0<\sigma_\e \le 1$ and $2< p< \infty$.
The case $d\ge 3$ is treated already in Remark \ref{re-P1}.
Assume $d=2$.
Then $\sigma_\e=\e |\ln \eta|^{1/2}$ and $\Phi_p(\eta^{-1} )= \eta^{\frac{2}{p}-1} |\ln \eta|^{-1}$.
Let $u_\e$ be the solution of \eqref{DP-0} with $F=0$.
By \eqref{P-5b} and  Theorem \ref{thm-S0},
\begin{equation}
\aligned
\|\nabla u_\e \|_{L^p(\Omega)}
 & \le C \e^{-1} \Phi_p(\eta^{-1}) \| u_\e \|_{L^p(\Omega)}
+ C \Phi_p(\eta^{-1}) \| f\|_{L^p (\Omega)}\\
& \le C \Phi_p(\eta^{-1}) \e^{-1} \sigma_\e \| f\|_{L^p(\Omega)}\\
&= C \eta^{\frac{2}{p}-1} |\ln \eta|^{-\frac12} \| f\|_{L^p(\Omega)},  
\endaligned
\end{equation}
where we have used  Theorem \ref{lemma-P1} for the second inequality.
\end{proof}

\begin{proof}[Proof of Theorem \ref{main-thm-2}]

Suppose $\sigma_\e\ge 1$ and $2<p<\infty$.
The case for $d\ge 2$ and $d<p<\infty$ is treated in Remark \ref{re-P2}.
Assume $d\ge 3$ and $2<p\le d$.
Let $u_\e$ be the solution of \eqref{DP-0} with $F=0$.
It follows from \eqref{P-5a} and Theorem \ref{thm-S0} that 
\begin{equation}
\aligned
\|\nabla u_\e \|_{L^p(\Omega)}
 & \le C \e^{-1} \eta^{\frac{d}{p}-1} \| u_\e \|_{L^p(\Omega)}
+ C \Phi_p (\eta^{-1}) \| f \|_{L^p(\Omega)}\\
& \le C \left(\e^{-1} \eta^{\frac{d}{p}-1} + \Phi_p(\eta^{-1}) \right) \| f\|_{L^p(\Omega)},
\endaligned
\end{equation}
where we have used Theorem \ref{lemma-P1} for the last inequality.
In view of \eqref{PhiR-4} we obtain 
$$
\| \nabla u_\e \|_{L^p(\Omega)} \le  \left(\e^{-1} \eta^{\frac{d}{p}-1} + 1  \right) \| f\|_{L^p(\Omega)},
$$
for $2<p<d$, and
$$
\| \nabla u_\e \|_{L^p(\Omega)} \le  \left(\e^{-1}  + |\ln \eta|^{1-\frac{1}{d}}  \right) \| f\|_{L^p(\Omega)}
$$
for $p=d$.
This gives the estimates in Theorem \ref{main-thm-2} for the case $2<p\le d$.
\end{proof}


\section{Estimates of $B_p(\Omega_{\e, \eta})$}\label{section-B}

In this section we establish upper bounds for $B_p(\Omega_{\e, \eta})$ in \eqref{W1p-0}.
We begin with the case $1<p\le 2$.
The estimate \eqref{uB-1} below  is sharp. See Theorem \ref{CPthm-6}. 

\begin{thm}\label{addin-thm}
    Suppose $1< p \leq 2$.
Let $\Omega$ be a bounded Lipschitz  domain in $\R^d$, $d\ge 2$,  and $\Omega_{\e, \eta}$ be given by \eqref{O-e}. Then 
 \begin{equation}\label{uB-1}
    \begin{aligned} 
   B_p(\Omega_{\e, \eta} ) \leq C \min (\sigma_\e, 1),
       \end{aligned}
     \end{equation}
where $C$ depends only on $d$, $p$, $\Omega$, and $\{Y_z^s\}$.
\end{thm}

\begin{proof}
This was proved in \cite{shen2023uniform}, using \eqref{P1-0} and a duality argument.
\end{proof}

Next, we consider the case $2< p< \infty$ and $0< \sigma_\e\le 1$.
The estimate \eqref{W1pB-01} below  is sharp for $d\ge 3$ and $2< p< \infty$.
See Theorem \ref{B-low-1}.

\begin{thm}\label{main-thm-3}
Suppose that $0< \sigma_\e \le 1$ and $2< p< \infty$.
Let $\Omega$ be a bounded $C^1$ domain in $\R^d$ and $\Omega_\e$ be given by \eqref{O-e}.
Then
\begin{equation}\label{W1pB-01}
B_p (\Omega_{\e, \eta} ) \le 
C \e \eta^{1-d + \frac{d}{p}}, 
\end{equation}
where $C$ depends only on $d$, $p$, $\Omega$, and $\{Y_z^s\}$.
\end{thm}

\begin{proof}

It follows from \eqref{P-2c} that
\begin{equation}\label{B-1a}
B_p(\Omega_{\e, \eta})
\le C \Phi_p(\eta^{-1}) \left\{ 1+ \e^{-1} \min (\sigma_\e, 1)\right\} ^2 \e.
\end{equation}
Using  $0< \sigma_\e \le 1$, we obtain 
\begin{equation}\label{B-1b}
B_p(\Omega_{\e, \eta}) \le C \Phi_p(\eta^{-1})\sigma_\e^2 \e^{-1}.
\end{equation}
In the case $d=2$, since $\Phi_p(\eta^{-1}) = \eta^{\frac{2}{p} -1} |\ln \eta|^{-1}$ and $\sigma_\e = \e  |\ln \eta|^{1/2}$, 
this gives \eqref{W1pB-01}.

In the case $d\ge 3$ and $p>d$, we have   $\sigma_\e = \e \eta^{1-\frac{d}{2}}$ and $\Phi_p(\eta^{-1})=\eta^{\frac{d}{p}-1}$.
It follows from \eqref{B-1b} that \eqref{W1pB-01} holds.
Since \eqref{W1pB-01} also holds for $p=2$,
by Riesz-Thorin Interpolation Theorem, the estimate \eqref{W1pB-01} holds for all $2< p< \infty$.
\end{proof}

Finally, we treat the case $2< p< \infty$ and $\sigma_\e \ge 1$. The estimate \eqref{W1pB-02}  below is sharp
in the case $d\ge 3$ and $2<p<d$.
See Theorem \ref{B-low-1}.

\begin{thm}\label{main-thm-4}
Suppose that $\sigma_\e \ge 1$ and $2< p< \infty$.
Let $\Omega$ be a bounded $C^1$ domain in $\R^d$ and $\Omega_{\e, \eta}$ be given by \eqref{O-e}.
 Then 
\begin{equation}\label{W1pB-02}
B_p(\Omega_{\e, \eta}) \le 
\left\{
\aligned
& C (1+ \e^{-1} \eta^{\frac{d}{p}-1}) & \quad & \text{ if }  d\ge 3 \text{ and } 2<p<d,\\
& C (\e^{-1} + |\ln \eta|^{1-\frac{1}{d}}) & \quad & \text{ if }   d\ge 3 \text{ and } p=d,\\
& C \e^{-1} \eta^{\frac{d}{p}-1} & \quad & \text{ if }  d\ge 3 \text{ and } d< p<\infty,\\
& C \e^{-1} \eta^{\frac{2}{p}-1} |\ln \eta|^{-1}
& \quad &\text{ if } d=2 \text{ and } 2< p< \infty,
\endaligned
\right.
\end{equation}
where $C$ depends only on $d$, $p$, $\Omega$, and $\{Y_z^s\}$.
\end{thm}

\begin{proof}

Using $\sigma_\e\ge 1$, we obtain from \eqref{B-1a} that 
\begin{equation}
B_p(\Omega_{\e, \eta})
\le C \Phi_p(\eta^{-1}) \e^{-1}.
\end{equation}
In view of \eqref{PhiR-4}, this gives \eqref{W1pB-02}  if $d\ge 2$ and $p>d$.
For the remaining case $d\ge 3$ and $2< p\le d$, 
 we let $u_\e\in W^{1, p}_0(\Omega_{\e, \eta})$ be a solution of
$-\Delta u_\e =F$ in $\Omega_{\e, \eta}$, where $F\in L^p(\Omega)$.
We extend $u_\e $ to $\R^d$ by zero.
It follows from \eqref{P-5a} and Theorem \ref{thm-S0} that 
\begin{equation}
\aligned
\|\nabla u_\e\|_{L^p(\Omega)}
 & \le C \e^{-1} \eta^{\frac{d}{p}-1} \| u_\e \|_{L^p(\Omega)}
+ C \Phi_p(\eta^{-1})  \| F \|_{L^p(\Omega)}\\
& \le  C\left(  \e^{-1}  \eta^{\frac{d}{p}-1} 
+\Phi_p(\eta^{-1} )  \right) \| F \|_{L^p(\Omega)},
\endaligned
\end{equation}
where we have used Theorem \ref{lemma-P1} for the last inequality.
In view of \eqref{PhiR-4} we obtain \eqref{W1pB-02} for $d\ge 3$ and $2< p\le d$.
\end{proof}


\section{Correctors: the case $d\ge 3$ }\label{section-C1}

This section is dedicated to constructing and establishing estimates for the corrector $\chi_{\e, \eta}$
in the case $d\ge 3$. 
The two-dimensional case will be handled separately  in the next section.

Let $\Omega_{\e, \eta}$ be given by \eqref{O-e}, where $T_{z, \eta} $  be given by \eqref{T}.
The corrector $\chi_{\e, \eta}$ for $\Omega_{\e, \eta} $ is defined in a piecewise manner, following an idea in \cite{allaire1991homogenization}.
 For $z\in \mathbb{Z}^d$, let $y_z=z+x_z$ and $Q_z = z+Q(0, 1)$.
If $\e Q_z \subset \Omega$,
we define 
\begin{equation}\label{chi-2}
\chi_{\e, \eta} (x) = \begin{cases}
1 & \text{in} \ \ \e Q_z  \setminus  \e B(y_z, 1/5), \\
\phi^z_*(\frac{x-\e  y_z}{\e \eta} ) & \text{in} \ \ \e B(y_z, 1/6 )   \setminus  \e T_{z, \eta}, \\
 0 & \text{in} \ \  \e T_{z, \eta},
\end{cases}
\end{equation}
where $\phi^z_* $ is a solution to the  exterior problem,
\begin{equation}\label{ext-2}
\begin{cases}
\begin{alignedat}{2}
-\Delta \phi^z_* &= 0 \qquad &&\text{in} \ \ \mathbb{R}^d \setminus \overline{Y^s_z}, \\
\phi^z_* &= 0 \qquad &&\text{on} \ \ \partial Y^s_z, \\
\phi^z_* &\rightarrow 1 \qquad &&\text{as} \ \ |x| \rightarrow \infty.
\end{alignedat}
\end{cases}
\end{equation}
In the region $\e B(y_z, 1/5) \setminus  \e B(y_z,1/6)$ we let $\chi_{\e, \eta} $ solve the following Dirichlet problem,
\begin{equation}\label{gap}
\begin{cases}
\begin{alignedat}{2}
-\Delta \chi_{\e, \eta}  &= 0 \qquad  &&\text{in} \ \ \e B(y_z, 1/5) \setminus \e  \overline{B(y_z, 1/6 )},  \\
\chi_{\e, \eta}  &= \phi_*^z \Big(\frac{x-\e y_z}{\e \eta}\Big) \qquad  &&\text{on} \ \ \partial  B(\e y_z,
\e /6),  \\
\chi_{\e, \eta}  &= 1 \qquad  &&\text{on} \ \ \partial   B(\e y_z,\e/5). 
\end{alignedat}
\end{cases}
\end{equation}
Finally, if $\e Q_z$ does not lie entirely in $\Omega$, we define $\chi_{\e, \eta}  =1$ in $\e Q_z$.
As a result, the corrector $\chi_{\e, \eta} $ is defined in $\mathbb{R}^d$, taking value $1$ in $\Omega^c$ and $0$ in the holes $\e T_{z, \eta}$
 inside $\Omega$.
The definition in \eqref{gap}  bridges the gap between the boundary data of the exterior problem and 1. Thus, we see that 
 $\chi_{\e, \eta}  \in H^1(\Omega)$, 
 $\chi_{\e, \eta}=1$ on $\partial \Omega$ and
 $\chi_{\e, \eta}=0$ on $\partial \Omega_{\e, \eta}\setminus \partial \Omega$.
 
It is known that for each $\phi_*^z$ that satisfies \eqref{ext-2}, 
\begin{equation}\label{est-2}
\begin{cases}
\begin{alignedat}{1}
\phi^z_*({x}) &= 1 -c^z_* |x|^{2-d} + O(|x|^{1-d}), \\
\nabla \phi^z_*({x}) &= - c^z_*\nabla (|x|^{2-d})+ O(|x|^{-d}),\\
\nabla^2 \phi^z_*({x}) &= - c^z_*\nabla^2 (|x|^{2-d})+ O(|x|^{-d-1}),
\end{alignedat}
\end{cases}
\end{equation}
as $|x| \to \infty$, 
where $c^z_* = C_d\mu^z_* $ is the Newtonian capacity of $Y_z^s$
with 
\begin{equation}\label{mu}
\mu^z_* = \int_{\partial Y^s_z} n \cdot \nabla \phi^z_* \, d \sigma
\end{equation}
and $C_d =  \frac{1}{(d-2)|\partial B(0,1)|}$. See \cite{amrouche1997dirichlet, verchota1984layer}.
The condition  \eqref{c-0} on $Y_z^s$ ensures that there exist $\mu_0, \mu_1>0$ such that 
$\mu_0\le \mu_*^z \le \mu_1$ for any $z\in \mathbb{Z}^d$.

\begin{lemma}\label{LP-2} 
Let $\chi_{\e, \eta} $ be defined as above. Then 
\begin{equation}
\| \chi_{\e, \eta}  -1\| _{L^p(\Omega_{\e, \eta})} \leq
\begin{cases}
    C\eta^{d-2} & \text{for } \ \ 1 \leq p <\frac{d}{d-2}, \\
    C\eta^{d-2}|\ln \eta|^{d-2} & \text{for} \ \ p =\frac{d}{d-2}, \\
    C\eta^{\frac{d}{p}} & \text{for} \ \ p > \frac{d}{d-2},
\end{cases}
\end{equation}
where $C$ does not depend on $\e$ or $\eta$.
\end{lemma}

\begin{proof} 
Suppose $\e Q_z\subset \Omega$.
Note that $\chi_{\e, \eta}  -1=0$  in  $\e Q_z\setminus B(\e y_z, \e/5) $ 
and is  harmonic in  $ B(\e y_z,\e/{5}) \setminus \overline{ B(\e y_z, \e/{6})} $. Hence,  by the maximum principle,
\begin{equation*}
    \begin{aligned}
\| \chi_{\e, \eta}  -1 \| _{L^\infty(B(\e y_z,\e/{5}) \backslash B(\e y_z, \e/{6}))} 
&\leq \| \chi_{\e, \eta}  -1 \| _{L^\infty(\partial B(\e y_z, \e/{6}))}.
 \end{aligned}
\end{equation*}
Since $\chi_{\e, \eta}  (x) =\phi_*^z ((x-\e y_z)/(\e \eta))$ on $\partial B(\e y_z, \e/6)$,  it follows from \eqref{est-2} that 
$$
\| \chi_{\e, \eta}  -1 \|_{L^\infty(B(\e y_z,\e/{5}) \backslash B(\e y_z, \e/{6}))} 
\le C \eta^{d-2}.
$$
Next, we note that 
\begin{alignat*}{1}
    \int_{B(\e y_z,\frac{\e }{6})\backslash \e T_{z, \eta} } |\chi_{\e, \eta}  -1|^p dx 
    &=\int_{B(\e y_z, \frac{\e}{6})\backslash \e T_{z, \eta} } |\phi^z_*\Big(\frac{x-\e y_z}{\e \eta }\Big) -1|^p \, dx \\ 
    &= \e^d \eta^{d } \int_{B(0,\frac{1}{6\eta})\backslash  Y^s_z} |\phi^z_*(y) -1|^p\,  dy.
    \end{alignat*}
   In view of  \eqref{est-2} we obtain 
    \begin{alignat*}{1} \int_{B(\e y_z,\frac{\e}{6})\backslash \e  T_{z, \eta} } |\chi_{\e, \eta}  -1|^p \, dx
    &\leq C  \e^d \eta^{d} \int_{B(0,\frac{1}{6\eta})\backslash  Y^s_z} \frac{dy }{|y|^{(d-2)p}}  \\ 
    &\leq C \e^d \eta^{d}  \int_c^\frac{1}{6\eta} r^{d-1-dp+2p}\, dr.
    \end{alignat*}
Thus, 
\begin{equation}
    \| \chi_{\e, \eta} -1\|^p _{L^p(\e Q_z\setminus \e T_{z, \eta} )} \leq \begin{cases}
         C\e^d\eta^{p(d-2)} & \text{for } \ \ 1 \leq p <\frac{d}{d-2}, \\
    C\e^d\eta^d|\ln \eta| & \text{for} \ \ p =\frac{d}{d-2}, \\
    C\e^d \eta^d & \text{for} \ \ p > \frac{d}{d-2}.
    \end{cases}
\end{equation}
Summing over all cubes  $\e Q_z$ in $\Omega$ gives the  desired result, as the number of such cubes is bounded by $C\e^{-d}$.
\end{proof}

\begin{remark}\label{harbdd-2} 
Suppose $\e Q_z \subset \Omega$. Consider $\xi (x) = \chi_{\e, \eta}  (\e x + \e y_z)-1 $. 
The function $\xi$  is harmonic in $B(0,1/5) \setminus \overline{ B(0,1 /6)}$ 
with $\xi = 0 $ on $\partial B(0,1/5) $ and $\xi = \phi^z_* (\eta^{-1} x )-1 $ on $\partial B(0,1 /6)$. 
It follows that in $B(0, 1/5)\setminus B(0, 1/6)$, 
 $$
  |\nabla \xi| \leq C \| \phi_*^z (\eta^{-1} x) -1 \|_{C^{1, 1} (\partial B(0, 1/6))}
   \leq C\eta^{d-2},
   $$  
   where  we have used \eqref{est-2} for the last inequality.
   This implies that 
$$
|\nabla \chi_{\e, \eta}  | \leq C \e^{-1} \eta^{d-2}
$$
  in  $B(\e y_z,\e/3)\backslash B(\e y_z,\e/4) $. 
\end{remark}

\begin{lemma}\label{LPGRAD-2}
  Suppose $\e Q_z \subset \Omega$. Then 
      \begin{equation} \label{est-ch-2}   
    \left( \fint_{\e Q_z}|\nabla \chi_{\e, \eta}  |^p \right)^{1/p} \leq \begin{cases}
C \e^{-1} \eta^{\frac{d-p}{p}} & \text{if } p > d', \\
C \e^{-1} \eta^{d-2}|\ln \eta|^{\frac{1}{p}} & \text{if } p =d',  \\
C\e^{-1} \eta^{d-2} & \text{if } p < d',  
    \end{cases}
    \end{equation}
    where $d' = \frac{d}{d-1}$ and the constant $C$ does not depend on $\e$ or $\eta$.
 \end{lemma}

 \begin{proof}
     We start by decomposing the integral along the sub-regions within the cube to obtain
 \begin{equation}\label{decomp-2}
       \int_{\e Q_z}|\nabla \chi_{\e, \eta}  |^p 
       =    \int_{B(\e y_z,\frac{\e}{5})\setminus B(\e y_z,\frac{\e}{6})}|\nabla \chi_{\e, \eta} |^p +
        \int_{B(\e y_z, \frac{\e}{6}) \setminus \e T_{z, \eta}  }|\nabla \chi_{\e, \eta} |^p. 
 \end{equation}
  By Remark \ref{harbdd-2}, 
  the first term in the  right-hand side of  \eqref{decomp-2}  is bounded by  $C \e^d (\e^{-1} \eta^{d-2})^p$.
For the second term, we have 
\begin{equation}\label{1stlemest-2}
\begin{alignedat}{1}
     \int_{B(\e y_z, \frac{\e}{6}) \backslash \e  T_{z, \eta}  }|\nabla \chi_{\e, \eta}  |^p 
      &= \int_{B(\e y_z, \frac{\e}{6}) \backslash \e  T_{z, \eta}  }|\nabla (\phi_*^z\Big(\frac{x- \e y_z}{\e \eta}\Big)) |^p\,  dx  \\ 
      &=  (\e \eta )^{-p} 
      \int_{B(\e y_z, \frac{\e }{6}) \backslash \e   T_{z, \eta}  }|\nabla \phi_*^z(\frac{x-\e y_z}{\epsilon\eta})|^p \, dx
       \\ &= (\e \eta)^{d-p}   \int_{B(0, \frac{1}{6\eta}) \backslash  Y_z^s }|\nabla \phi_*^z(y)|^p \, dy\\
       &\le C  (\e \eta)^{d-p}
       \int_c^{\frac{1}{6\eta}}  r^{ (1-d) p + d-1} dr,
\end{alignedat}
\end{equation}
where we have used \eqref{est-2} as well as the fact that  $\nabla \phi_*^z$ is $L^p$ integrable near $\partial Y_z^s$ under the
condition that $Y_z^s$ is uniformly  $C^1$.
It follows that 
\begin{equation}\label{scdtrm-2}
    \left( \fint_{ \e B(y_z, 1/6)}|\nabla \chi_{\e, \eta}  |^p \right)^{1/p} \leq
    \begin{cases}
    C\e^{-1} \eta^{\frac{d-p}{p}} & \text{if } p > d' ,\\ 
    C \e^{-1} \eta^{\frac{d-p}{p}} | \ln \eta|^{\frac{1}{p} } & \text{if } p = d' , \\ 
    C \e^{-1} \eta^{d-2} & \text{if } p < d'.
    \end{cases}
\end{equation}
Combining  \eqref{scdtrm-2} with  the estimate for $B(\e y_z, \e/5) \setminus B(\e y_z, \e/6)$
gives \eqref{est-ch-2}. 
\end{proof}

Recall that $\sigma_\e = \e \eta^{-\frac{d-2}{2}}$ for $d\ge 3$ and $\mu_*^z$ is given by \eqref{mu}.

\begin{lemma}\label{lemma-3c}
Suppose $\e Q_z \subset \Omega$.
Then
\begin{equation}\label{est-3c}
\left |\int_{\e Q_z} \nabla \chi_{\e, \eta} \cdot \nabla \phi
- \int_{\e Q_z} \sigma_\e^{-2} \mu_*^z \phi \right|
\le C \e^{-1} \eta^{d-2} 
\int_{\e Q_z} |\nabla \phi|,
\end{equation}
where $\phi \in H^1(\e Q_z)$ and $\phi=0$ in $\e T_{z, \eta}$.
\end{lemma}

\begin{proof}
Note that 
\begin{equation}\label{ints-4}
\aligned
&   \int_{\e Q_z} \nabla \chi_{\e, \eta}  \cdot \nabla \phi - \int_{\e Q_z} \sigma_\e^{-2} \mu_*^z  \phi\\
 & = \int_{\e Q_z \setminus B(\e y_z, \e/6)}\nabla \chi_{\e, \eta}  \cdot \nabla \phi 
+ \int_{B(\e y_z,\e /6)\setminus \e T_{z, \eta} }\nabla \chi_{\e, \eta}  \cdot \nabla \phi 
-  \int_{\e Q_z} \sigma_\e^{-2}  \mu_*^z   \phi.
\endaligned
\end{equation}
By Remark \ref{harbdd-2}, 
\begin{equation}\label{FT-2}
\left|\int_{\e Q_z \setminus  B(\e y_z, \e/6)}\nabla \chi_{\e, \eta}  \cdot \nabla \phi \right| \leq C \e^{-1} \eta^{d-2}\int_{\e Q_z}  |\nabla \phi| .
\end{equation}
For the remaining two terms in the right-hand side of  \eqref{ints-4}, 
we use  integration  by parts to get 
\begin{equation}\label{secterm-2}
\begin{alignedat}{1}
\int_{\partial B(\e y_z,\e/6)} \left(\frac{\partial \chi_{\e, \eta} } {\partial n } - \fint_{\partial B(\e y_z,\e/6)}\frac{\partial \chi_{\e, \eta}}{\partial n }\right)
\cdot \left(\phi - \alpha\right) 
&+ \fint_{\partial B(\e y_z,\e/6)}\frac{\partial \chi_{\e, \eta} }{\partial n } \int_{\partial B(\e y_z,\e/6)}\phi \\ 
&- \int_{\e Q_z} \sigma_\e^{-2} \mu_*^z \phi, 
\end{alignedat}
\end{equation}
where $n$ denotes the outward unit normal  and
$\alpha$ is a constant to be determined.
Here we also use the fact that $\chi_{\e, \eta}$ is harmonic in $B(\e y_z, \e/6)\setminus \e T_{z, \eta}$.
 Using \eqref{est-2}, we obtain 
\begin{equation}\label{3-11}
\begin{alignedat}{1}
 & \left| \int_{\partial B(\e y_z,\e/6)} \left(\frac{\partial \chi_{\e, \eta} }{\partial n} 
 - \fint_{\partial B(\e y_z,\e/6)}\frac{\partial \chi_{\e, \eta} }{\partial n }\right)\cdot \left(\phi - \alpha\right)\right| \\ 
 &\leq 
C \e^{-1} \eta^{d-2}
 \int_{\partial B(\e y_z,\e/6)}|\phi -\alpha| \\  
 &\leq C\e ^{-1} \eta^{d-2} \left(\frac{1}{\e} \int_{B(\e y_z,\e/6)} |\phi - \alpha|  + \int_{B(\e y_z,\e/6)} |\nabla \phi|\right),
\end{alignedat}
\end{equation}
where we have used a trace inequality for the last step (see Remark \ref{re-trace}).
Picking $\alpha = \fint_{B(\e y_z,\epsilon/6)}\phi$ and applying the Poincar\'e inequality 
gives the desired bound. 

To bound 
\begin{equation}\label{N-2}
\left| \fint_{\partial B(\e y_z,\e/6)}\frac{\partial \chi_{\e, \eta} }{\partial n } \int_{\partial B(\e y_z,\e/6)}\phi- \int_{\e Q_z} \sigma_\e^{-2} \mu_*^z  \phi \right|, 
\end{equation} 
 we move the average onto the integral of $\phi$ and note that 
\begin{equation*}
\begin{alignedat}{1}
 \int_{\partial B(\e y_z,\e/6) } \frac{\partial \chi_{\e, \eta} }{\partial n } &= (\e \eta)^{-1}
   \int_{\partial B(\e y_z,\e/6) } \frac{\partial \phi^z_*}{\partial n }\Big (\frac{x-\e y_z}{\e \eta}\Big) d\sigma(x)\\ 
   &= (\e \eta)^{d-2}\int_{\partial B(0,\frac{1}{6\eta}) } \frac{\partial \phi^z_*}{\partial  n } d\sigma(y) \\ 
 &= (\e \eta)^{d-2} \int_{\partial Y^s_z} \frac{\partial \phi^z_*}{\partial n }d\sigma(y)
 \\ &=\mu_*^z(\e \eta)^{d-2} .
 \end{alignedat}
 \end{equation*}
As a result, \eqref{N-2} becomes 
$$
 \mu_*^z (\e \eta)^{d-2} \left | \fint_{\partial B(\e y_z,\e/4)}\phi - \fint_{\e Q_z}   \phi \right|, 
 $$
 which, in view of \eqref{trace-2},  is bounded by 
 $$ 
 C \e^{-1}\eta^{d-2} \int_{\e Q_z} |\nabla \phi|.
 $$
This completes the proof.
\end{proof}

\begin{remark}\label{re-trace}
In \eqref{3-11} we have used the trace inequality,
\begin{equation}\label{trace}
\int_{\partial B(x_0,  r)} 
|\phi |
\le \frac{d}{r} \int_{B(x_0, r)} |\phi |
+  \int_{B(x_0, r)} |\nabla \phi|,
\end{equation}
for $\phi \in H^1(B(x_0, r))$.
This follows by writing 
$$
\int_{\partial B(x_0, r)} |\phi | =\frac{1}{r} \int_{\partial B(x_0, r)} |\phi| ((x-x_0)\cdot n) \, d\sigma(x) 
$$
and applying the divergence theorem.
Replacing $\phi$ in \eqref{trace} by $\phi -\alpha$, where $\alpha =\fint_{\e Q_z} \phi$,  we obtain 
\begin{equation}\label{trace-1}
\int_{\partial B(\e y_z ,  c\e )} |\phi -\fint _{\e Q_z} \phi |
\le C \int_{\e Q_z } |\nabla \phi|,
\end{equation}
where we have used a Poincar\'e  inequality,
\begin{equation}\label{P-1}
\int_{\e Q_z} | \phi -\fint_{\e Q_z} \phi |
\le C \e  \int_{\e Q_z} |\nabla \phi|.
\end{equation}
It follows from \eqref{trace-1} that
\begin{equation}\label{trace-2}
\left|
\fint_{\partial B(\e y_z, c\e)} \phi 
-\fint_{\e Q_z} \phi \right|
\le C \e \fint_{\e Q_z} |\nabla \phi|.
\end{equation}
\end{remark}


\section{Correctors: the two-dimensional case}\label{section-C2}

In this section we introduce a corrector $\chi_{\e, \eta}$ for the case $d=2$ and establish its estimates analogous to those in Lemmas
\ref{LP-2}, \ref{LPGRAD-2} and \ref{lemma-3c}.

For $z\in \mathbb{Z}^2$, let  $\phi^z_* $ be  a  solution to the  exterior problem,
\begin{equation}\label{ext-2a}
\begin{cases}
\begin{alignedat}{2}
-\Delta \phi^z_* &= 0 \qquad &&\text{in} \ \ \mathbb{R}^2 \setminus \overline{Y^s_z}, \\
\phi^z_* &= 0 \qquad &&\text{on} \ \ \partial Y^s_z, \\
\phi^z_*(x)  - \ln |x|   & = O(1) \qquad &&\text{as} \ \ |x| \rightarrow \infty.
\end{alignedat}
\end{cases}
\end{equation}
It is known that 
\begin{equation}\label{asy-2}
\left\{
\aligned
\nabla \phi_*^z (x) & = \frac{x}{ |x|^2} + O(|x|^{-2}),\\
\nabla^2\phi_*^z (x)  & = O(|x|^{-2}),
\endaligned
\right.
\end{equation}
as $|x| \to \infty$.
Recall  that  $y_z = z + x_z$ and $Q_z= z+Q(0, 1) $.
If $\e Q_z \subset \Omega$, we let 
\begin{equation}\label{chi-3}
\chi_{\e, \eta} (x) = \begin{cases}
1 & \text{in} \ \ \e Q_z  \setminus  \e B(y_z, 1/5), \\
 \phi^z_*(\frac{x-\e  y_z}{\e \eta} ) / |\ln \eta| & \text{in} \ \ \e B(y_z, 1/6 )   \setminus  \e T_{z, \eta}, \\
 0 & \text{in} \ \  \e T_{z, \eta}.
\end{cases}
\end{equation}
In the region $\e B(y_z, 1/5) \setminus  \e B(y_z,1/6)$,  we let $\chi_{\e, \eta} $ solve the following Dirichlet problem,
\begin{equation}\label{gap-2}
\begin{cases}
\begin{alignedat}{2}
-\Delta \chi_{\e, \eta}  &= 0 \qquad  &&\text{in} \ \ \e B(y_z, 1/5) \setminus \e  \overline{B(y_z, 1/6 )},  \\
\chi_{\e, \eta}  &= \phi_*^z \Big(\frac{x-\e y_z}{\e \eta}\Big) / |\ln  \eta| \qquad  &&\text{on} \ \ \partial  B(\e y_z,
\e /6),  \\
\chi_{\e, \eta}  &= 1 \qquad  &&\text{on} \ \ \partial   B(\e y_z,\e/5). 
\end{alignedat}
\end{cases}
\end{equation}
Finally, if  $\e Q_z $  is not contained in $\Omega$,
 let $\chi_{\e, \eta} = 1$ in $\e Q_z$, as in the case of $d\ge 3$.
By construction, we have  $\chi_{\e, \eta}  \in H^1(\Omega)$. Moreover,  $\chi_{\e, \eta} =1$ in $\Omega^c$ and $\chi_{\e, \eta}=0$ in 
$\Omega\setminus \Omega_{\e, \eta}$.
Furthermore, in the region $\e B(y_z, 1/5)\setminus B(y_z, 1/6)$, we have $L^\infty$ estimates,
\begin{equation}\label{uni-2}
| \chi_{\e, \eta}  -1|  \le C |\ln \eta|^{-1}
\quad \text{ and } \quad
|\nabla \chi_{\e, \eta} |\le C \e^{-1}  |\ln \eta|^{-1}.
\end{equation}
As in the case $d\ge 3$, the estimates in \eqref{uni-2} follow
from  the maximum principle and the Lipschitz estimates for harmonic functions.

\begin{lemma}\label{LPD2-2} 
Let $\chi_{\e, \eta} $ be defined as above. Then for $1 \le  p < \infty$, 
\begin{equation}\label{est-ch-3}
\| \chi_{\e, \eta}  -1\| _{L^p(\Omega_{\e, \eta})} \leq C|\ln \eta |^{-1}, 
\end{equation}
where $C$ does not depend on $\e$ or $\eta$.
\end{lemma}

\begin{proof}
    
    Note that  if $\e Q_z\subset \Omega$, 
    \begin{equation*}
    \begin{aligned}
     \int_{\e Q_z\setminus \e T_{z, \eta}}
      |\chi_{\e, \eta}  -1|^p  & \le C \e^{2} |\ln \eta|^{-p} + 
        \int_{B(\e y_z, \e/6 )\setminus \e T_{z, \eta}} \left| \frac{\ln| (x-\e y_z)/\e| }{ \ln\eta} \right|^p\, dx \\
   &\le  \frac{ C\e ^2 }{|\ln \eta |^p  }\left\{ 
    \int_{B(0,1 )}  |\ln |  x| |^p \, dx  + 1 \right\}\\
        &\leq \frac{ C \e^2}{| \ln \eta|^p},
    \end{aligned}
    \end{equation*}
    where we have used \eqref{uni-2} as well as the fact that $\phi_*^z(x)=\ln |x| + O(1)$ as $|x|\to \infty$.
Summing over all such cubes yields
$$
 \| \chi_{\e, \eta}  -1 \| ^p_{L^p(\Omega_{\e, \eta })} \leq C|\ln \eta |^{-p},
$$
which yields \eqref{est-ch-3}.
\end{proof}

Recall that $\sigma_\e =\e |\ln \eta|^{1/2}$ for $d=2$.

\begin{lemma}\label{LPGRADD2-2}
  Suppose $\e Q_z \subset \Omega$. Then 
      \begin{equation}    \label{est-ch-4}
    \left(\fint_{\e  Q_z} |\nabla \chi_{\e, \eta} |^p \right)^{1/p} \leq \begin{cases}
    C\sigma_\e^{-1}|\ln\eta|^{-1/2} & \text{if } 1 < p < 2, 
      \\ 
      C\sigma_\e^{-1} & \text{if } p = 2,  \\ 
       C \sigma_\e^{-1} \eta^{\frac{2-p}{p}}|\ln\eta|^{-1/2} & \text{if } 2 < p < \infty,
    \end{cases}
     \end{equation}
    where the constant $C$ does not depend on $\e$ or $\eta$.
 \end{lemma}
 
\begin{proof}
    Note  that 
    \begin{equation*}
\begin{aligned}
\int_{\e Q_z} |\nabla \chi_{\e, \eta} |^p 
&\le C \e^{2-p} |\ln \eta|^{-p} + \int_{B(\e y_z, \e/6)\setminus  \e T_{z, \eta}} |\nabla \chi_{\e, \eta} |^p dx \\
&=C \e^{2-p} |\ln \eta|^{-p} +
 \int_{B(\e y_z, \e/6)\setminus  \e T_{z, \eta} }  \left| \left( \frac{C}{|\ln\eta| |x- \e y_z|}\right)\right|^p dx \\
&= C\e^{2-p} | \ln \eta|^{-p} \left\{ 1+ 
 \int_{c\eta}^1 r^{1-p} dr \right\},
\end{aligned}
    \end{equation*}
    where we have used \eqref{uni-2} and  \eqref{asy-2}.
    It follows that 
\begin{equation*}
    \begin{aligned}
 \left(\fint_{\e Q_z} |\nabla \chi_{\e, \eta} |^p \right)^{1/p}   \leq  \begin{cases}
     C\e^{-1} |\ln\eta|^{-1} & \text{if } 1 < p < 2 ,
      \\ 
      C\e^{-1}|\ln\eta|^{-\frac12} & \text{if } p = 2, \\ 
       C \e^{-1} \eta^{\frac{2-p}{p}}|\ln\eta|^{-1} & \text{if } 2 < p < \infty, 
    \end{cases}
\end{aligned}
    \end{equation*}
  which gives \eqref{est-ch-4}
  \end{proof}
 
\begin{lemma}\label{lemma-2c}
Suppose $\e Q_z \subset \Omega$. Then
\begin{equation}\label{est-2c}
\left | 
\int_{\e Q_z} \nabla \chi_{\e, \eta} \cdot \nabla \phi
-    2\pi \sigma_\e^{-2} 
\int_{\e Q_z} \phi \right|
\le  C \e^{-1} |\ln \eta|^{-1}   \int_{\e Q_z} |\nabla \phi |,
\end{equation}
where $\phi \in H^1(\e Q_z)$ and $\phi=0$ in  $ \e T_{z, \eta}$.
\end{lemma}

\begin{proof} 

The proof is similar to that of Lemma \ref{lemma-3c}.
Note that 
   \begin{equation}\label{std2-2}
    \begin{aligned}
&       \int_{\e Q_z}   \nabla \chi_{\e, \eta}  \cdot \nabla \phi 
     - 2\pi \sigma_\e^{-2}  \int_{\e Q_z}   \phi \\
&      =
 \int_{ B(\e y_z, \e/5)\setminus B(\e y_z,  \e /6)} 
     \nabla \chi_{\e, \eta}  \cdot \nabla \phi 
     + \int_{B(\e y_z, \e /6) \setminus \e T_{z, \eta}} \nabla \chi_{\e, \eta} \cdot \nabla \phi
     -  \frac{2 \pi}{   |\ln \eta | } \fint_{\e Q_z}   \phi \\ 
     &= \int_{ B(\e y_z, \e/5)\setminus B(\e y_z,  \e /6)} 
     \nabla \chi_{\e, \eta}  \cdot \nabla \phi 
+ \left\{ \int_{\partial B(\e y_z, \e/6) } \frac{\partial \chi_{\e, \eta} }{\partial n  } \phi 
     - \frac{2\pi}{  |\ln \eta |} \fint_{\e Q_z}  \phi\right\} \\
     &= I_1 +I_2,
         \end{aligned}
 \end{equation}
where we have used  integration  by parts and the fact that $\chi_{\e, \eta}$ is harmonic in $B(\e y_z, \e/6) \setminus \e T_{z, \eta} $.
By \eqref{uni-2},  it is easy to see that $|I_1|$ is bounded by the right-hand side of \eqref{est-2c}.

To bound $I_2$, we write it as 
\begin{equation}\label{secterm-2a}
\begin{alignedat}{1}
I_2 = & \int_{\partial B(\e y_z,\e/6)} \left(\frac{\partial \chi_{\e, \eta} } {\partial n } - \fint_{\partial B(\e y_z,\e/6)}\frac{\partial \chi_{\e, \eta}}{\partial n }\right)
\cdot \left(\phi - \alpha\right) \\
&\qquad +\left\{  \int_{\partial B(\e y_z,\e/6)}\frac{\partial \chi_{\e, \eta} }{\partial n } \fint_{\partial B(\e y_z,\e/6)}\phi 
-\frac{2\pi}{| \ln \eta|}  \fint_{\e Q_z}  \phi\right\} \\
&=I_{21} +I_{22}, 
\end{alignedat}
\end{equation}
where  $\alpha$ is a constant to be determined. Using \eqref{asy-2}, one may show that
\begin{equation}
\aligned
|I_{21}|
& \le \frac{C}{\e |\ln \eta|} \int_{\partial B(\e y_z, \e/6)} |\phi-\alpha|\\
& \le \frac{C}{\e |\ln \eta|} \int_{\e Q_z} |\nabla \phi|, 
\endaligned
\end{equation}
where we have choose $\alpha$ to be the average of $\phi$ over $\e Q_z$, as in the proof of Lemma \ref{lemma-3c}.

Finally, to handle $I_{22}$, we note that by \eqref{asy-2},
$$
\int_{\partial B(\e y_z, \e/6)} \frac{\partial \chi_{\e, \eta}}{\partial n}
=\frac{1}{|\ln  \eta|} \left\{ 2 \pi + O(\eta )  \right\}.
$$
This, together with \eqref{trace-2}, shows that
$$
\aligned
|I_{22}|
 & \le \frac{C}{\e |\ln \eta| } \int_{\e Q_z} |\nabla \phi|
+\frac{C \eta}{\e |\ln \eta|  } \int_{\e Q_z} |\phi|\\
& \le C \frac{C}{\e |\ln \eta| } \int_{\e Q_z} |\nabla \phi|,
\endaligned
$$
where we have used Lemma \ref{lemma-P0} with $p=1$  for the last inequality.
This completes the proof.
\end{proof}


\section{Convergence Rates}\label{section-Con}

In this section we study  the Dirichlet problem with nonhomogeneous boundary conditions, 
\begin{equation}\label{NH-2}
\begin{cases}
\begin{alignedat}{2}
- \Delta u_\e &= 0 \qquad &&\text{in} \ \ \Omega_{\e, \eta}, \\
u_\e &= h \qquad &&\text{on} \ \ \partial \Omega, \\
u_\e &= 0 \qquad &&\text{on} \ \ \partial\Omega_{\e, \eta} \setminus  \partial \Omega,
\end{alignedat}
\end{cases}
\end{equation}
where $\Omega_{\e, \eta} $ is given in \eqref{O-e} and $h \in H^{1/2}(\partial\Omega)$.
Let $V_\e (x)=  V(x/\e)$, where
\begin{equation}\label{V}
V(y)= \sum_{z\in \mathbb{Z}^d} 
\mu_*^z   \chi_{Q(z, 1)} (y),
\end{equation}
and $\mu_*^z$ is given by \eqref{mu} for $d\ge 3$ and $\mu_*^z =2 \pi  $ for $d=2$.
In \eqref{V} we have used  $\chi_{Q(z, 1)}$ to denote  the characteristic function of the cell $Q(z, 1)$.
Note that $c_0 \le V (y) \le c_1$ for some $c_0, c_1>0$.
Let $\chi_{\e, \eta}$ be the corrector constructed in the last two sections  for the domain $\Omega_{\e, \eta}$.
We will show that the solution $u_\e$ of \eqref{NH-2}  is well approximated by $\chi_{\e, \eta} u_{0, \e}$, where $u_{0, \e}$ is
 the solution  of  a boundary value problem in $\Omega$,
 \begin{equation}\label{Shro-2}
 \left\{
\aligned
- \Delta u_{0,\e} + \sigma_\e^{-2} V_\e(x)u_{0,\e}  & = 0 &\quad &  \text{in} \ \ \Omega, \\
u_{0,\e}  & = h & \quad &  \text{on} \ \ \partial \Omega,
\endaligned
\right.
\end{equation}
 which we will call the intermediate problem, 
  for a Schr\"odinger operator $-\Delta + \sigma_\e^{-2} V_\e$.  
  The parameter $\sigma_\e$, which depends on $\e$ and $\eta$,  is given by \eqref{sigma}.
 The rest of this section will be divided into two parts. 
The first part deals with   the case  $d \geq 3$, while   the second treats the case $d =2$.


\subsection{Dimension $d \geq 3$}

The goal of this subsection is to prove the following.

 \begin{thm}
    \label{cor2-2}
 Let $u_\e$ be a solution of \eqref{NH-2} and $u_{0,\e}$ a solution of \eqref{Shro-2}.
  Let 
  \begin{equation}\label{r-e}
  r_{\e}  = u_\e - \chi_{\e, \eta}  u_{0,\e},
  \end{equation}
   where $\chi_{\e, \eta} $ is the corrector constructed in Section \ref{section-C1}.
  Assume that $u_{0,\e} \in W^{1,p}(\Omega)$ for some $2<p< \infty$.
  Then for $d \geq 3$, 
 \begin{equation}\label{con-3d0}
\| \nabla r_\e\| _{L^2(\Omega_{\e, \eta})}\leq  \begin{cases}
C \eta^\frac{d-2}{2}\left(\sigma_\e^{-1} \| u_{0,\e}\| _{L^p( \Omega)
} + \| \nabla u_{0,\e}\| _{L^p( \Omega)}\right)  & \text{if} \ \ p \geq d, \\
C  \eta^{d/q} \left( \sigma_\e^{-1} \| u_{0,\e}\| _{L^p( \Omega)
} + \| \nabla u_{0,\e}\| _{L^p( \Omega)}\right) 
& \text{if} \ 2< p < d,
\end{cases}
\end{equation}
where $\frac12 =\frac{1}{p}+\frac{1}{q}$  and  $C $ does not depend on $\e$ or $\eta$.
\end{thm}

Note that  $\chi_{\e, \eta} \in W^{1, p}(\Omega)$ for any $p>2$. Moreover, 
$\chi_{\e, \eta}=1$ on $\partial \Omega$ and $\chi_{\e, \eta}=0$ on $\partial\Omega_{\e, \eta}\setminus \partial \Omega$.
It  follows that $r_\e \in H_0^1(\Omega_{\e, \eta})$.
The proof of Theorem \ref{cor2-2} will rely heavily on the estimates for $\chi_{\e, \eta} $ in Section \ref{section-C1}.

\begin{lemma}\label{GRAD-2}  
 Let $\phi \in H_0^1(\Omega_{\e, \eta} )$. Then 
\begin{equation}\label{est-3c1}
 \left| \int_{\Omega_{\e, \eta}} \nabla \chi_{\e, \eta} \cdot  \nabla \phi 
 - \int_{\Omega_{\e, \eta}} \sigma_\e^{-2} V_\e \phi \right| \leq C\e^{-1} \eta^{d-2}  \int_{\Omega_{\e, \eta}}|\nabla \phi|, 
 \end{equation}
where $V_\e = V(x/\e)$ and $V(y)$  is defined in \eqref{V}.
\end{lemma}

\begin{proof}
We extend $\phi$ to $\mathbb{R}^d$ by zero.
Note that 
\begin{equation}\label{gdest1-2}
\aligned
 &\left|  \int_{\Omega_{\e, \eta}} \nabla \chi_{\e, \eta}  \cdot \nabla \phi  
 -\int_{\Omega_{\e, \eta} } \sigma_\e^{-2} V_\e \phi \right| \\
 & \le \sum_{z\in \Omega'_{\e} } \left |  \int_{\e Q_z } \nabla \chi_{\e, \eta}  \cdot \nabla \phi  
 -\int_{\e Q_z}  \sigma_\e^{-2} V_\e \phi\right | +  \left | \int_\mathcal{F} \sigma_\e^{-2} V_\e \phi \right|,
\endaligned
\end{equation}
where $\Omega'_{\e}$ is the union of cubes $\e Q_z$ lying entirely inside of $\Omega$ and $\mathcal{F}$ is the union 
of cubes with non-empty intersection with the boundary of $\Omega$. 
By Lemma \ref{lemma-3c}, the first term in the right-hand side of \eqref{gdest1-2} is bounded by 
the right-hand side of \eqref{est-3c1}.
For the second term   we have 
\begin{equation*}
\begin{alignedat}{1}
\left| \int_{\mathcal{F}} 
\sigma_\e^{-2} V_\e \phi \right| 
&\leq C\sigma_\e^{-2} 
\int_{\text{\rm dist(} x,\partial \Omega)\leq C\e} |\phi|\\ 
&\leq C\e \sigma_\e^{-2}  \int_{\text{\rm dist}( x,\partial \Omega)
\leq C\e} |\nabla \phi| 
\leq C \e^{-1} \eta^{d-2}\int_{\Omega_{\e, \eta }} |\nabla \phi|, 
\end{alignedat}
\end{equation*}
where we have used a Poincar\'e-Sobolev inequality for the second inequality.
\end{proof} 

\begin{proof}[Proof of Theorem \ref{cor2-2}]

Let $u_\e$ be a solution of  \eqref{NH-2} and $u_{0,\e}$ a solution of  \eqref{Shro-2}.
Assume further that $u_{0,\e} \in W^{1, p} (\Omega)$ for some $p>2$.
 Let $r_\e = u_\e - \chi_{\e, \eta}  u_{0,\e} $.
Note that for any  $ \phi\in H_0^1(\Omega_{\e, \eta})$,
\begin{equation}\label{preibp-2}
\begin{alignedat}{1}
\int_{\Omega_{\e, \eta}} \nabla r_\e \cdot \nabla \phi 
&= 
  - \int_{\Omega_{\e, \eta}}  (\nabla \chi_{\e, \eta}  \cdot  \nabla \phi ) u_{0,\e} 
  - \int_{\Omega_{\e, \eta}} \chi_{\e, \eta}  \left( \nabla u_{0,\e} \cdot \nabla \phi \right) \\ 
&= - \int_{\Omega_{\e, \eta}} \nabla \chi_{\e, \eta}  \cdot \nabla( u_{0,\e} \phi) 
+\int_{\Omega_{\e, \eta}}[ \nabla (\chi_{\e, \eta} -1) \cdot \nabla u_{0,\e}] \phi  \\
& \qquad \qquad \qquad  - \int_{\Omega_{\e, \eta}} (\chi_{\e, \eta}  -1) \left( \nabla u_{0,\e} \cdot \nabla \phi \right) 
- \int_{\Omega_{\e, \eta}}\left( \nabla u_{0,\e} \cdot \nabla \phi \right).
\end{alignedat}
 \end{equation}
 By using integration  by parts,  the right-hand side of \eqref{preibp-2} becomes 
 \begin{equation*}
 \begin{aligned}
 - \int_{\Omega_{\e, \eta}} \nabla \chi_{\e, \eta} \cdot  \nabla (u_{0,\e}\phi) 
 &- 2\int_{\Omega_{\e, \eta}} (\chi_{\e, \eta}  -1) \left( \nabla u_{0,\e} \cdot \nabla \phi \right) \\ 
  &- \int_{\Omega_{\e, \eta}} (\chi_{\e, \eta}  -1) \Delta u_{0,\e} \cdot \phi +  \int_{\Omega_{\e, \eta}}\left( \Delta u_{0,\e} \cdot  \phi \right), 
 \end{aligned}
 \end{equation*}
which, in  view of  \eqref{Shro-2}, equals to 
\begin{equation*}
\begin{aligned}
 - \int_{\Omega_{\e, \eta}} \nabla \chi_{\e, \eta} \cdot  \nabla (u_{0,\e }\phi) + \sigma_\e^{-2} 
 \int_{\Omega_{\e, \eta}} V_\e  u_{0,\e} \phi 
 &- 2\int_{\Omega_{\e, \eta}} (\chi_{\e, \eta}  -1) \left( \nabla u_{0,\e} \cdot \nabla \phi \right) 
 \\
 &- \sigma_\e^{-2} \int_{\Omega_{\e, \eta} } (\chi_{\e, \eta}  -1) V_\e  u_{0,\e} \cdot \phi.
\end{aligned}
 \end{equation*}
 Note that $V_\e(x) \leq c_2$. We now apply Lemma \ref{GRAD-2} to obtain 
 \begin{equation*}
 \begin{aligned}
  \left| \int_{\Omega_{\e, \eta} } \nabla r_\e \cdot \nabla \phi \right|
 &  \leq C\e^{-1}\eta^{d-2} \int_{\Omega_{\e, \eta}} |\nabla (u_{0,\e} \phi)| 
  +2 \int_{\Omega_{\e, \eta} }  |\chi_{\e, \eta}  -1|
  | \nabla u_{0,\e} |  |\nabla \phi |  \\ 
  & \qquad +  C \sigma_\e^{-2} \int_{\Omega_{\e, \eta} } |\chi_{\e, \eta}  -1|  |u_{0,\e} |  | \phi|  \\
&   \leq  C\e^{-1}\eta^{d-2} 
\int_{\Omega_{\e, \eta}} |\nabla u_{0,\e}| | \phi| 
+C\e^{-1}\eta^{d-2} \int_{\Omega_{\e, \eta}} |\nabla \phi| | u_{0,\e} |\\ 
&\qquad
   + 2\int_{\Omega_{\e, \eta}} |\chi_{\e, \eta}  -1| 
    | \nabla u_{0,\e}|   |\nabla \phi |
     + C \sigma_\e^{-2} \int_{\Omega_{\e, \eta}} | \chi_{\e, \eta}  -1|  |u_{0,\e} |  |\phi|.
 \end{aligned}
 \end{equation*}
Applying H\"older's inequality yields  
\begin{equation}\label{con-10}
\begin{aligned}
  \left|  \int_{\Omega_{\e, \eta} } \nabla r_\e \cdot \nabla \phi \right|
  \leq &C\sigma_\e^{-1} \eta^{\frac{d-2}{2}}\| \nabla \phi\|_{L^2(\Omega_{\e, \eta})}||u_{0,\e}\|_{L^2(\Omega_{\e, \eta})} \\
&\qquad  + C\sigma_\e^{-1} \eta^{\frac{d-2}{2}}\| \phi\|_{L^{2}(\Omega_{\e, \eta})}\| \nabla u_{0,\epsilon}\|_{L^{2}(\Omega_{\e, \eta})} \\ 
  & \qquad \qquad
  +C\| \chi_{\e, \eta}  -1\|_{L^q(\Omega_{\e, \eta})} \| \nabla \phi\|_{L^2(\Omega_{\e, \eta})} \|\nabla u_{0,\e}\|_{L^p(\Omega_{\e, \eta})} \\ 
  & \qquad \qquad\qquad
  + C\sigma_\e^{-2} \| \chi_{\e, \eta}  -1\|_{L^q(\Omega_{\e, \eta})} \| \phi\|_{L^2(\Omega_{\e, \eta})} \| u_{0,\e}\| _{L^p(\Omega_{\e, \eta})},
    \end{aligned}
    \end{equation} 
where $\frac{1}{2} = \frac{1}{q} + \frac{1}{p }$. 
By applying the  Poincar\'e inequality $\|\phi\|_{L^2(\Omega_{\e, \eta})} \le C \sigma_\e \|\nabla \phi \|_{L^2(\Omega_{\e, \eta})}$ in \eqref{con-10}
 we have 
$$
\aligned
\left|  \int_{\Omega_{\e, \eta} } \nabla r_\e \cdot \nabla \phi \right|
&\leq C \eta^{\frac{d-2}{2}}\|\nabla \phi \|_{L^2(\Omega_{\e, \eta})}\left( \sigma_\e^{-1} \|u_{0,\e}\|_{L^2(\Omega)} 
+ \| \nabla u_{0,\e}\|_{L^{2}(\Omega)}\right) \\ 
&+C\| \chi_{\e, \eta}  -1\| _{L^q(\Omega_{\e, \eta})} ||\nabla \phi\|_{L^2(\Omega_{\e, \eta})} \left(\| \nabla u_{0,\e}\| _{L^p(\Omega)}  
+ \sigma_\e^{-1} \| u_{0,\e} \|_{L^p(\Omega)}\right).
 \endaligned
 $$
Choosing $\phi = r_\e$,  we obtain 
\begin{equation*}
    \| \nabla r_\e\| _{L^2(\Omega_{\e, \eta} )} 
    \leq C\left(\eta^{\frac{d-2}{2}}+\| \chi_{\e, \eta}  -1\| _{L^q(\Omega_{\e, \eta})}\right)\left(\| \nabla u_{0,\e}||_{L^{p} (\Omega)} 
    + \sigma_\e^{-1} \| u_{0,\e}\|_ {L^p(\Omega)} \right).
\end{equation*}
Applying Lemma \ref{LP-2} gives  \eqref{con-3d0}.
\end{proof}


\subsection{Dimension $d=2$}

We now consider the case  $d =2$.

\begin{thm}\label{versionD2-2}
Let $u_\e$ be a solution of \eqref{NH-2} amd $u_{0,\e}$ a solution of \eqref{Shro-2}
 Assume that $u_{0,\e} \in W^{1,p}(\Omega)$ for some $2<p<\infty$. 
 Let $r_\e = u_\e - \chi_{\e, \eta}  u_{0,\e} $. Then
 \begin{equation}\label{rate-2}
    \| \nabla r_\e\|_{L^2(\Omega_{\e, \eta} )} 
    \leq C | \ln \eta|^{-1/2}\left( \sigma_\e^{-1} \| u_{0,\e}\| _{L^p( \Omega)}
+ \| \nabla u_{0,\e} 
\|_{L^p( \Omega)}\right),
\end{equation}
where  $C $ does not depend on $\e$ or $\eta$.
\end{thm}

The following lemma plays the role of Lemma \ref{GRAD-2}.

\begin{lemma}\label{GRADD2-2}  
Let $\chi_{\e, \eta} $ be the corrector constructed in Section \ref{section-C2}.
Then, for any  $\phi \in H^1_0(\Omega_{\e, \eta} )$,
\begin{equation}\label{est-2c3}
\left| \int_{\Omega_{\e, \eta}} \nabla \chi_{\e, \eta} \cdot  \nabla \phi -
 2\pi \sigma_\e^{-2} \int_{\Omega_{\e, \eta}}  \phi \right| 
\leq \frac{C}{\e |\ln \eta|}
 \int_{\Omega_{\e, \eta}}|\nabla \phi|.
\end{equation}
\end{lemma}

\begin{proof}

The proof is similar to that of Lemma \ref{GRAD-2}.
    Let $\phi \in H_0^1(\Omega_{\e, \eta})$ and extend it to $\mathbb{R}^2$ by zero.
    The left-hand side of \eqref{est-2c3} is bounded by
    \begin{equation}\label{2c31}
    \sum_{z} \left|
    \int_{\e Q_z} \nabla \chi_{\e, \eta} \cdot \nabla \phi 
    -2\pi \sigma_\e^{-2} \int_{\e Q_z} \phi \right|
    +\left| \int_{\mathcal{F}}  2\pi \sigma_\e^{-2}  \phi \right|,
    \end{equation}
    where the sum is taken over those $z$'s for which $\e Q_z\subset \Omega$ and
    $\mathcal{F}$ is the union of cubes $\e Q_z$ with non-empty intersections with $\partial \Omega$.
    As in the proof of Lemma \ref{GRAD-2}, the second term in \eqref{2c31} can be handled
    readily by a Poincar\'e inequality, since $\phi=0$ on $\partial \Omega$.
     For the first term we use Lemma \ref{lemma-2c} to bound it by the right-hand side of \eqref{est-2c3}.
      \end{proof}

\begin{proof}[Proof of Theorem \ref{versionD2-2}]
Let $u_\e$ be a solution of \eqref{NH-2} and $u_{\e, 0}$  a solution of  \eqref{Shro-2}.
  Assume further that $u_{0,\e} \in W^{1, p}(\Omega)$. 
  Let  $r_\e = u_\e - \chi_{\e, \eta}  u_{0,\e} $.
    As in the proof of Theorem \ref{cor2-2}, 
   for $\phi\in H^1_0(\Omega_{\e, \eta})$,  we have 
\begin{equation*}
\begin{aligned}
\int_{\Omega_{\e, \eta}} \nabla r_\e \cdot \nabla \phi 
&=  - \int_{\Omega_{\e, \eta}} \nabla \chi_{\e, \eta}  \cdot \nabla (u_{0,\e}\phi)
+  2 \pi \sigma_\e^{-2} \int_{\Omega_{\e, \eta}}  u_{0,\e} \phi \\
 &\qquad - 2\int_{\Omega_{\e, \eta}} (\chi_{\e, \eta}  -1) 
\left( \nabla u_{0,\e} \cdot \nabla \phi \right) 
 -2\pi \sigma_\e^{-2} 
\int_{\Omega_{\e, \eta} } (\chi_{\e, \eta}  -1) u_{0,\e}  \phi.
    \end{aligned}
\end{equation*}
By applying  Lemma \ref{GRADD2-2} we obtain 
\begin{equation*}
\begin{aligned}
\left| \int_{\Omega_{\e, \eta}} \nabla r_\e \cdot \nabla \phi \right|
&\leq \frac{C}{\e |\ln \eta|}  \int_{\Omega_{\e, \eta}} |\nabla (u_{0,\e} \phi)|
   +  2 \int_{\Omega_{\e, \eta} } | \chi_{\e, \eta}  -1|  | \nabla u_{0,\e}|  |\nabla \phi |\\
&   \qquad + 2\pi  \sigma_\e^{-2} \int_{\Omega_{\e, \eta} } | \chi_{\e, \eta} -1|  |  u_{0,\e}|   | \phi|\\
&=I.
   \end{aligned}
\end{equation*}
Next, 
by H\"older's inequality, 
\begin{equation*}
    \begin{aligned} 
   I
 &\leq C\sigma_\e^{-1} | \ln \eta|^{-\frac{1}{2}}\left(\| \nabla \phi\| _{L^2(\Omega_{\e, \eta})}\| u_{0,\e}\|_{L^2(\Omega_{\e, \eta})} 
 +\|  \phi \|_{L^2(\Omega_{\e, \eta})}\| \nabla u_{0,\e}\|_{L^2(\Omega_{\e, \eta})}\right) \\
 &\quad  +C\| \chi_{\e, \eta}  -1 \| _{L^q(\Omega_{\e, \eta})} \left(\| \nabla \phi\|_{L^2(\Omega_{\e, \eta})} \|
 \nabla u_{0,\e}\|_{L^p(\Omega_{\e, \eta})}  +
  \sigma_\e^{-2} \|  \phi\|_{L^2(\Omega_{\e, \eta})} \| u_{0,\e}\|_{L^p(\Omega_{\e, \eta})}\right),
    \end{aligned}
\end{equation*}
where $\frac{1}{2} = \frac{1}{p} + \frac{1}{q}$. By applying the Poincar\'e inequality
$ \| \phi \|_{L^2(\Omega_{\e, \eta}) }\le C \sigma_ \e  \|\nabla \phi \|_{L^2(\Omega_{\e, \eta})}$ and Lemma \ref{LPD2-2}, we obtain 
\begin{equation*}
\begin{aligned}
  I
 \le C  |\ln \eta|^{-\frac12} 
\|\nabla \phi \|_{L^2(\Omega_{\e, \eta})}
\left(\| \nabla u_{0,\e}\|_{L^p(\Omega)}  
+ \sigma_\e^{-1}\| u_{0,\e}\|_{L^p(\Omega)}\right).
\end{aligned}
\end{equation*}
\medskip
By choosing $\phi= r_\e$, we obtain \eqref{rate-2}.
\end{proof}

 
 \section{An Intermediate Problem}\label{section-I}
 
  In  this section we consider the boundary value problem for a Schr\"odinger operator,
 \begin{equation}\label{IVP}
 \left\{
 \aligned
 -\Delta u  +\lambda^2  V(x)  u  &=  F & \quad & \text{ in } \Omega, \\
 u& =g & \quad & \text{ on } \partial \Omega,
 \endaligned
 \right.
 \end{equation}
 where $\lambda>0$ and  $V=V(x)$ is a potential satisfying the condition $0< \mu_0\le V \le \mu_1$.
 
 \begin{lemma}\label{lemma-S1}
 Let $\Omega$ be a bounded Lipschitz  domain in $\mathbb{R}^d$ and $\lambda>0$.
 Suppose that  $u\in H^1(B(x_0,2 r) \cap \Omega)$ and 
 \begin{equation}\label{S-10}
 \left\{
 \aligned
 -\Delta u +\lambda^2 V(x)  u &  = 0 & \quad & \text{ in }  B(x_0, 2r) \cap \Omega, \\
 u& =0 & \quad & \text{ on }  B(x_0,2 r) \cap \partial \Omega,
 \endaligned
 \right.
 \end{equation}
 where $x_0\in \partial\Omega$ and $0<r< r_0$.
 Then
 \begin{equation}\label{S-11}
 \sup_{B(x_0, r)\cap \Omega} |u|
 \le C \left( \fint_{B(x_0,2 r) \cap \Omega} |u|^2 \right)^{1/2},
 \end{equation}
 where $C$ depends on $d$, $\Omega$ and $(\mu_0, \mu_1)$.
 Moreover, 
   \begin{equation}\label{S-12}
 \left(\fint_{B(x_0, r)\cap \Omega} |\nabla u|^p \right)^{1/p}
 \le C \left( \fint_{B(x_0, 2r) \cap \Omega} |\nabla u|^2 \right)^{1/2},
 \end{equation}
where $2< p< 3+\delta$ for $d\ge 3$,  $2< p< 4+ \delta$ for $d=2$, 
and  $\delta>0$ depends on $d$ and $\Omega$.
 The constant  $C$ in  \eqref{S-12}  depends on $d$, $p$, $\Omega$ and $(\mu_0, \mu_1)$.
 If $\Omega$ is a bounded $C^1$ domain, the estimate \eqref{S-12} holds for any $2< p< \infty$.
 \end{lemma}
 
 \begin{proof}
 
 Let $\phi \in C_0^\infty(B(x_0, 2r)) $.
 Then
$$
\int_{B(x_0, 2r)  \cap \Omega   } \nabla u \cdot \nabla (u \phi^2) + \int_{B(x_0, 2r) \cap \Omega } \lambda^2 V  u \cdot u\phi^2 = 0.
$$
It follows  by Cauchy's inequality that
$$
\int_{B(x_0, 2r)\cap \Omega } |\nabla u|^2  \phi^2 
+ \int_{B(x_0, 2r) \cap \Omega  } \lambda^2 V u^2 \phi^2 \leq  C \int_{B(x_0, 2r) \cap \Omega  } |u|^2 |\nabla \phi|^2.
 $$
By choosing a suitable  cut-off function $\phi$, we see that 
$$ 
\lambda^2 \int_{B(x_0, tr)\cap \Omega}    u^2 \leq \frac{C}{r^2 (s-t)^2} \int_{B(x_0, sr)\cap \Omega } u^2
$$
for  $ 1\le  t < s\le 2 $.
This implies that 
\begin{equation}\label{iter-3}
    \int_{B(x_0, tr)\cap \Omega } u^2
    \leq \frac{C}{ 1+ (s-t)^2 \lambda^2 r^2  }\int_{B(x_0, sr)\cap \Omega  } u^2.
\end{equation}
 By iterating \eqref{iter-3} $k$ times we obtain 
 \begin{equation}\label{Qt-3}
     \left( \fint_{ B(x_0, tr)\cap \Omega }  u^2\right)^{1/2} \leq \frac{C_{k, t, s} }{(1+ \lambda r)^{2k}} 
     \left( \int_{ B(x_0, sr) \cap \Omega  }  u^2\right)^{1/2} 
 \end{equation}
 for any $k \ge 1$, where $1\le  t< s\le 2$.
 
Now, since 
$$
 \Delta u  = \lambda^2 V u \ \ \ \text{in} \  B(x_0, 2r)\cap \Omega 
 \quad 
 \text{ and } \quad u =0 \quad \text{ on } B(x_0, 2r) \cap \partial \Omega,
$$
the boundary $L^\infty$  estimates for  Laplace's equation in Lipschitz  domains  give
\begin{equation*}
\begin{aligned}
\left(\fint_{B(x_0, tr) \cap \Omega } |u|^p\right)^{1/p} 
&\leq  C_{t, s}  \left(\fint_{B(x_0, sr)\cap \Omega }| u|^2  \right)^{1/2} + C_{t, s} r^2\lambda^2 \left(\fint_{B(x_0, sr)\cap \Omega } | u|^q \right)^{1/q},
\end{aligned}
\end{equation*}
 where $ 0< \frac{1}{q} -\frac{1}{p} < \frac{2}{d}$, and
 \begin{equation*}
\begin{aligned}
\sup_{B(x_0, tr)\cap \Omega} |u|
&\leq  C_{t, s}  \left(\fint_{B(x_0, sr)\cap \Omega }| u|^2  \right)^{1/2} + C_{t, s} r^2\lambda^2 \left(\fint_{B(x_0, sr)\cap \Omega } | u|^q \right)^{1/q},
\end{aligned}
\end{equation*}
where $q>(d/2)$.
 These two estimates, together with \eqref{Qt-3},  yield \eqref{S-11} by a bootstrapping argument.

 Note that the $W^{1, p}$ estimate for Laplace's equation in $C^1$ domains \cite{MR1331981} gives
$$
\left( \fint_{B(x_0, ts) \cap \Omega } |\nabla u|^p \right)^{1/p} 
\leq C_{t, s}  \left(\fint_{B(x_0, sr)\cap \Omega } |\nabla u|^2  \right)^{1/2} 
+ C_{t, s} r \lambda^2 \left( \fint_{B(x_0, sr) \cap \Omega  } |u|^q \right)^{1/q},
 $$
where $1<t<s<2$ and  $0< \frac{1}{q}-\frac{1}{p} < \frac{1}{d} $. 
If $\Omega$ is a Lipschitz domain,
we need to impose the additional  conditions that 
$2<p< 3+\delta$ for $d\ge 3$, and $2<p< 4+\delta$ for $d=2$, where $\delta>0$ depends on $d$ and $\Omega$.
Thus,  by \eqref{Qt-3} and  \eqref{S-11}, 
\begin{equation}\label{nabu-3}
\begin{aligned}
\left( \fint_{B(x_0, r) \cap \Omega } |\nabla u|^p \right)^{1/p} 
&\leq C  \left(\fint_{B(x_0, 2r) \cap \Omega } |\nabla u|^2  \right)^{1/2} 
+ Cr^{-1}  \left(\fint_{B(x_0, 2r)\cap \Omega  } |u|^2 \right)^{1/2} \\ 
 &\leq  C \left(\fint_{B(x_0, 2r) \cap \Omega } |\nabla u|^2  \right)^{1/2}, 
 \end{aligned}
 \end{equation}
  where we have used a Poincar\'e inequality and the fact that $u=0$ on $B(x_0, 2r) \cap \partial \Omega$ for the last step.
\end{proof}

\begin{remark}\label{re-S1}
Suppose that $u \in H^1(B(x_0, 2r))$ and
$-\Delta u + \lambda^2 V u =0$ in $B(x_0, 2r)$. Then
 \begin{equation}\label{S-11a}
 \sup_{B(x_0, r)} |u|
 \le C \left( \fint_{B(x_0,2 r)}   |u|^2 \right)^{1/2},
 \end{equation}
   \begin{equation}\label{S-12a}
 \sup_{B(x_0, r) } |\nabla u|
 \le C \left( \fint_{B(x_0, 2r) } (  |\nabla u|^2 + \lambda^2 |u|^2 )  \right)^{1/2},
 \end{equation}
where $C$ depends on $d$ and $(\mu_0, \mu_1)$.
The proof is similar to that of Lemma \ref{lemma-S1}.
\end{remark}

An operator $T$ is called sublinear if there exists $K\ge 1$ such that
\begin{equation}\label{sublinear}
| T(f+g)| \le K \left( | T(f)| + |T(g)|\right).
\end{equation}
The  following theorem was proved in \cite{shenbounds}.

\begin{thm}\label{RV1-3}
Let $\Omega$ be a bounded Lipschitz domain in $\mathbb{R}^d$. Let $T$ be a bounded sublinear operator on $L^2(\Omega)$ with  
\begin{equation}\label{L2L2-3}
 \| T\|_{L^2 \rightarrow L^2} \leq C_0.
 \end{equation}
Let $q > 2 $. Suppose that 
\begin{equation}\label{RHT-3}
\left( \fint_{\Omega \cap B(x_0, r) } |T(f)|^q \right)^{1/q} \leq N  \left( \fint_{\Omega \cap B(x_0, 2r) } |T(f)|^2 \right) ^{1/2}
\end{equation}
for any ball $B(x_0, r)  $ with the property that $0< r< r_0$ 
and either $B(x_0, 4r) \subset \Omega $ or $x_0 \in \partial \Omega$ and 
for any $ f \in C^\infty_0(\Omega) $ with supp$(f) \subset \Omega \backslash B(x_0, 4r) $.
 Then for any $F \in L^p (\Omega) $,
$$ 
\| T(F)\|_{L^p(\Omega)} \leq C  \| F \|_{L^p(\Omega)},
 $$
where $2 < p < q$ and $C$ depends  at most  on $d, p,q,C_0,N, r_0,  \Omega $ and $K $.
\end{thm}

\begin{thm}\label{corint-3}
 Let $\Omega$ be a bounded Lipschitz   domain in $\mathbb{R}^d$. 
  Let $u\in W^{1, 2}_0 (\Omega)$ be a solution of 
 \begin{equation}\label{shrorhs-3}
     \left\{
     \aligned
    -\Delta u +  \lambda^2  V(x)   u   & = F & \quad & \text{ in }   \Omega, \\ 
    u  & =0 &  \quad & \text{ on }  \partial \Omega,
\endaligned
\right.
\end{equation}
where $F \in L^2(\Omega)$. Suppose $F\in L^p(\Omega)$, where $2<p< 3+\delta$ for $d\ge 3$, and 
$2< p< 4+\delta $ for $d=2$.
Then
\begin{equation}\label{S2-0}
 \lambda \| \nabla u \|_{L^p(\Omega)} + \lambda^ 2  \| u \|_{L^p(\Omega)} \leq C \| F \|_{L^p(\Omega)}, 
\end{equation}
where $C$ depends on $d$, $p$, $\Omega$ and $(\mu_0, \mu_1)$.
If $\Omega$ is a bounded $C^1$ domain, the estimate \eqref{S2-0} holds for any $2<p< \infty$.
\end{thm}

\begin{proof}

To prove  \eqref{S2-0},  we consider the sublinear operator,
    $$
    T_\lambda  (F)  = \lambda  |\nabla u| + \lambda^2 |u|,
    $$
where $u$ is a solution of \eqref{shrorhs-3},     and apply Theorem \ref{RV1-3}. 
    Note that by the energy estimate, $\| T_\lambda  \|_{L^2(\Omega) \to L^2(\Omega)} \le C$,
    where $C$ depends on $d$, $\Omega$ and $(\mu_0, \mu_1)$.
    To verify the condition \eqref{RHT-3},
     let $f \in C_0^\infty (\Omega)$ such that $f=0$ in  $B(x_0, 2r)$, where $0< r< r_0$ and either $x_0\in \partial \Omega$ or 
     $B(x_0, 4r)\subset \Omega$.
     Let $w$ be the solution of \eqref{shrorhs-3} with $f$ in the place of $F$. Then $-\Delta w + \lambda^2 V w =0$ in $B(x_0, 4r) \cap \Omega$ and
     $w=0$ on $B(x_0, 4r) \cap \partial \Omega$ (if $x_0 \in \partial \Omega$).
     If $\Omega$ is a bounded $C^1$ domain, 
     it follows from Lemma \ref{lemma-S1} and Remark \ref{re-S1} that 
     \begin{equation}\label{S2-1}
     \left(\fint_{B(x_0, r)\cap \Omega} ( \lambda |\nabla w| + \lambda^2  |w|  )^q \right)^{1/q}
     \le C \left(\fint_{B(x_0, 2r) \cap \Omega}
     ( \lambda |\nabla w |+ \lambda^2  | w|)^2 \right)^{1/2}
     \end{equation}
for any $2< q< \infty$,
     where $C$ depends on $d$, $q$, $\Omega$, and $(\mu_0, \mu_1)$.
     As a result, by Theorem \ref{RV1-3}, we obtain 
     $ \| T_\lambda (F) \|_{L^p (\Omega)} \le C \| F \|_{L^p (\Omega)}$ for any $2< p< \infty$.
     If $\Omega$ is a bounded Lipschitz domain, 
     the estimate \eqref{S2-1} holds for $2<q<3+\delta$  if $d\ge 3$, and for $2<q<4+\delta$ if $d=2$.
     As a result, we obtain \eqref{S2-0} for $2< p< 3+\delta$ if $d\ge 3$, and for $2< p< 4+\delta$ if $d=2$,
     where $\delta>0$ depends on $d$ and $\Omega$.
\end{proof}

\begin{thm}\label{extshro-3}
    Let $\Omega$ be a bounded Lipschitz domain in $\mathbb{R}^d$ with connected boundary.
      Let  $u \in W^{1,2}(\Omega)$ be a solution of 
$$
\left\{
\aligned
-\Delta u + \lambda^2 V(x) u  &  = 0 & \quad &  \text{ in } \Omega,  \\
u  & = g & \quad & \text{ on } \partial \Omega,
\endaligned
\right.
$$
where $g \in W^{1, 2}(\partial \Omega)$.
Then
\begin{equation}\label{S3-0}
\left\{
\aligned
\| u \|_{L^p(\Omega)}  & \le C \| g \|_{L^2(\partial \Omega)},\\
\| \nabla u\|_{L^p(\Omega)}   
 & \leq C \left\{ \| \nabla_t  g\| _{L^2(\partial \Omega)} + \lambda  \| g\| _{L^2(\partial \Omega) } \right\}, 
\endaligned
\right.
\end{equation}
where $p = \frac{2d}{d-1}$, $\nabla_t g$ denotes the tangential gradient of $g$ on $\partial\Omega$,
and $C$ depends on $d$, $\Omega$ and $(\mu_0, \mu_1)$.
\end{thm}

\begin{proof}
    We start by solving the Dirichlet problem, 
\begin{equation*}
\begin{cases}
\begin{alignedat}{2}
    -\Delta G &= 0  \qquad &&\text{in} \ \Omega,\\
    G &= g  \qquad &&\text{on} \ \partial \Omega.
\end{alignedat}
\end{cases} 
\end{equation*}
By the well known nontangential-maximal-function estimates for harmonic functions in Lipschitz domains (see e.g. \cite{verchota1984layer}),
$$
\| ( \nabla G)^*  ||_{L^2(\partial \Omega )} \leq C 
\| \nabla_t g \|_{L^2(\partial \Omega) }  \quad  \text{ and } \quad  \| ( G)^* \|_{L^2(\Omega) } \leq \| g\| _{L^2(\partial \Omega)},
$$
where $\nabla_t g$ denotes the tangential  gradient of $g$ on $\partial \Omega$.
It follows that 
$$
\aligned
& \| \nabla G \|_{L^p(\Omega)}
\le C \| (\nabla G)^* \|_{L^2(\partial \Omega)}
\le C \| \nabla_t g \|_{L^2(\partial \Omega)}, \\
& \| G \|_{L^p(\Omega)}
\le C \| ( G)^* \|_{L^2(\partial \Omega)}
\le C \|g \|_{L^2(\partial \Omega)}, \\
\endaligned
$$
where $p=\frac{2d}{d-1}$ and we have used the inequality $\| w \|_{L^p(\Omega)} \le C \| (w)^* \|_{L^2(\partial \Omega)}$
for functions $w$ in $\Omega$.

Next, let  $v = u-G$.
 Then $v$ satisfies the following equation,
$$
 \left\{
 \aligned
    -\Delta v + \lambda^2 V(x) v  & = F  & \quad &  \text{ in } \Omega, \\ 
    v  & =0 &\quad &  \text{ on }  \partial \Omega,
\endaligned
\right.
$$
where $F  = - \lambda^2  V(x) G $. 
Let $p=\frac{2d}{d-1}$.
Note that $p=4$ for $d=2$, $p=3$ for $d=3$, and $p<3$ for $d\ge 4$.
In view of Theorem  \ref{corint-3} we have 
\begin{equation*}
\lambda^2  \| u-G\| _{L^p(\Omega)} 
\leq C\|  F \| _{L^p(\Omega)} \leq C \lambda^2  \| G\|_{L^p(\Omega)}.
\end{equation*}
This implies 
\begin{equation}\label{f1-3}
\| u\|_{L^p(\Omega)} \leq C\| G\|_{L^p(\Omega)} \leq C\| g\| _{L^2(\partial \Omega)}  
\end{equation}
for $p = \frac{2d}{d-1}$.
Moreover, by Theorem  \ref{corint-3},
$$
 \lambda \| \nabla(u-G)\| _{L^p(\Omega)} 
 \leq C \| \lambda^2  V  G \|_{L^p(\Omega)}.
 $$
It follows that
\begin{equation}\label{f2-3}
    \begin{aligned}
        \| \nabla u\|_{L^p(\Omega)} 
        &\leq \| \nabla G\|_{L^p(\Omega)} 
        + C \lambda  \| G\|_{L^p(\Omega)} \\ 
        &\leq C  \| \nabla_t g\| _{L^2(\partial \Omega)} 
        + C \lambda \| g\| _{L^2(\partial \Omega)}
          \end{aligned}
\end{equation}
for $p = \frac{2d}{d-1}$. 
\end{proof}

 
\section{Large-scale $W^{1, p}$ estimates }\label{section-L}

Let $u_{\e}$ be a solution of 
\begin{equation}\label{BVP-L}
\left\{
\aligned
-\Delta u_\e & =F +\text{\rm div} (f) & \quad & \text{ in } \Omega_{\e, \eta},\\
u_\e & =0 & \quad & \text{ on } \partial \Omega_{\e, \eta}, 
\endaligned
\right.
\end{equation}
where $\Omega_{\e, \eta}$  is given by \eqref{O-e}.
Let 
\begin{equation*}
S_\e (F, f)=\left(\fint_{Q(x, 2\e) } |\nabla u_\e|^2 \right)^{1/2},
\end{equation*}
where we have extended $u_\e$ to $\mathbb{R}^d$ by zero.
Note that  
$$
\| S_\e (F, f) \|_{L^2(\mathbb{R}^d)} = \| \nabla u_\e \|_{L^2(\Omega_{\e, \eta})}.
$$
It follows that 
\begin{equation}\label{L-1}
\|S_\e (F, f) \|_{L^2(\mathbb{R}^d)}\le 
 \| f \|_{L^2(\Omega_{\e, \eta})} +  C \min (\sigma_\e, 1) \| F \|_{L^2(\Omega_{\e, \eta})}.
\end{equation}
This section is dedicated to the $L^p$ estimates for $S_\e(F,f)$ in Theorem \ref{thm-S0}.

\begin{thm}\label{thm-L1}
Let $\Omega$ be a bounded $C^1$ domain in $\R^d$ and $\Omega_{\e, \eta} $ be given by \eqref{O-e}.
Then,
\begin{equation}\label{Large-e}
\| S_\e (0,f)\|_{L^p(\Omega)}
\le C \| f \|_{L^p(\Omega)}
\end{equation}
 for $2< p< \infty$,
where $C$ depends only on $d$, $p$,  $\Omega$, and $\{Y_z^s\}$.
\end{thm}

In the following we will assume $0< \eta< \eta_0$, where $\eta_0>0$ is sufficiently small. 
The case $\eta\ge \eta_0$  with any fixed $\eta_0>0$ is the same as the case $\eta=1$, which is  contained in \cite{MR2111721}.
To prove Theorem \ref{thm-L1},  we use Theorem \ref{RV1-3}. The $L^2$ boundedness condition \eqref{L2L2-3} is given by \eqref{L-1}.
To verify the reverse H\"older condition \eqref{RHT-3}, it suffices to show that if $\Delta u_\e =0$ in $Q(x_0, 4R) \cap \Omega_{\e, \eta}$ and
$u_\e =0 $ on $Q(x_0, 4R) \cap \partial \Omega_{\e, \eta}$, where  $0<R< c_0$ and
either $Q(x_0, 4R)\subset \Omega$ or $x_0 \in \partial \Omega$, then
\begin{equation}\label{L-2}
\left(\fint_{Q(x_0, R)\cap \Omega} \left(\fint_{Q(x, 2\e)} |\nabla u_\e|^2 \right)^{p/2}  dx \right)^{1/p}
\le C \left(\fint_{Q(x_0, 2R)\cap \Omega}
\left( \fint_{Q(x, 2\e)} |\nabla u_\e|^2 \right) dx  \right)^{1/2}
\end{equation}
for $p>2$.
It is  not hard to see that
\eqref{L-2} holds for $0< R< C \e$.
The proof of \eqref{L-2} for the large-scale  case $R\ge C \e$ uses a real-variable argument and relies on  the estimates
from the previous two sections.

We begin with the interior case $Q(x_0, 4R) \subset \Omega$.
The following theorem was proved in \cite[Theorem 3.2]{MR2353255}. 

\begin{thm}\label{RV-L}
  Let $F \in L^2(Q(x_0, 2R) ) $ and $2<p< q$. 
  Suppose that for each cube $Q=Q(y, r)$ with $y\in Q(x_0, R)$ and $0<r<c_0 R $,
  there exists two functions $R_Q$ and $F_Q$ such that
  \begin{equation}\label{L-R1}
  |F| \le |R_Q| + |F_Q| \quad \text{ in } 2Q,
  \end{equation}
  \begin{equation}\label{L-R2}
  \left( \fint_{2Q} |R_Q|^q\right)^{1/q}
  \le N \left( \fint_{8Q} |F|^2 \right)^{1/2},
  \end{equation}
  \begin{equation}\label{L-R3}
  \left(\fint_{2Q} |F_Q|^2 \right)^{1/2}
  \le \delta
  \left(\fint_{8Q} |F|^2 \right)^{1/2},
  \end{equation}
          where $N > 1$ and $0 < c_0 < 1 $. 
          Then there exists $\delta_0 > 0 $,  depending only on $d, N,  c_0, p, q$,  with the property that if $0 \leq \delta < \delta_0 $,
           then $F \in L^p(Q(x_0, R)) $ and 
    \begin{equation}\label{L-R4}
    \left(\fint_{Q(x_0, R)} |F|^p \right)^{1/p}
    \le C \left(\fint_{Q(x_0, 2R)} |F|^2 \right)^{1/2},
        \end{equation}
    where $C $ depends at most on $d, c_0, p, q, N$. 
\end{thm}

We use $Q_{\e, \eta} (x_0, r)$ to denote $\Omega_{\e, \eta}$, where $\Omega=Q(x_0, r)$.
We point out that  in general, $ Q_{\e, \eta}(x_0, r) 
\neq    Q(x_0, r)\cap \Omega_{\e, \eta} $, as there is no hole near the boundary  of $Q(x_0, r) $ in 
$Q_{\e, \eta} (x_0, r)$.
However, under the assumption $r\ge C \e$, it is possible to find $Q(y_0, t)$ such that 
$Q(x_0, r) \subset Q(y_0, t)\subset Q(x_0, 2r)$ and   $Q_{\e, \eta} (y_0, t)=\Omega_{\e, \eta} \cap Q(y_0, t)$.
This observation is used in the proof of the following lemma.

\begin{lemma}\label{lemma-L1}
Let $u \in H^1(Q(x_0, 4r))$, where  $r\ge 8\e$ and $Q(x_0, 8r)\subset \Omega$. 
Suppose that $\Delta u=0$ in $Q(x_0, 4r) \cap \Omega_{\e, \eta} $ and
$u =0 $ on $Q(x_0, 4r)\setminus  \Omega_{\e, \eta}$.
Then there exists $v\in H^1(Q(x_0, 2r))$ such that
\begin{equation}\label{L-1-1}
\left(\fint_{Q(x_0, r)} |\nabla (u-v)|^2\right)^{1/2}
\le C \psi   (\eta)
\left(\fint_{Q(x_0, 3r)} |\nabla u |^2 \right)^{1/2},
\end{equation}
\begin{equation}\label{L-1-2}
\max_{x\in Q(x_0, r)} 
\left(\fint_{Q(x, 2\e) } |\nabla v|^2 \right)^{1/2} 
\le C \left(\fint_{Q(x_0, 3r)} |\nabla u|^2 \right)^{1/2},
\end{equation}
where  $\psi (\eta)=\eta^{1/2} $ for $d\ge 3$ and $\psi(\eta) =|\ln \eta|^{-1/2}$ for $d=2$.
The constant $C$ depends only on $d$ and $\{ Y_z^s\} $.
\end{lemma}

\begin{proof}
Since $r\ge 8 \e$, without loss of generality, we may assume that  $x_0=0$ and $r= 2^j \e$ for some
$j\ge 2$.
By  dilation  we may also  assume $r=1$.
Note that the parameter $\eta$ remains invariant under dilation.
Since $u=0$ on $ Q(0, 3)\setminus \Omega_{\e, \eta} $,  it follows by Lemma \ref{lemma-P0} that 
\begin{equation}\label{L-1-1a}
\sigma_\e^{-2}  \int_{Q(0, 3)} 
|u|^2
\le C \int_{Q(0, 3)} |\nabla u|^2.
\end{equation}
We claim that there exists $t\in [2, 3]$ such that
$Q_{\e, \eta} (0, t) = Q(0, t) \cap \Omega_{\e, \eta}$ and
\begin{equation}\label{L-1-3}
\sigma_\e^{-2}
\int_{\partial Q(0, t)}
|u|^2 + \int_{\partial Q(0, t)} |\nabla u|^2
\le C \int_{Q(0, 3)} |\nabla u|^2.
\end{equation}
To see \eqref{L-1-3}, we consider the set $E= \{ t\in [2, 3]: | t -k \e |\le c_0 \e \text{ for some } 1\le k \le 2^{j+2}  \}$.
Note that   $|E|\ge c>0$ and
$Q_{\e, \eta}(0, t) = Q(0, t) \cap \Omega_{\e, \eta} $ for $t\in E$.
If $  \eqref{L-1-3}$ fails for all $t\in E$, we may integrate the left-hand side of \eqref{L-1-3}
with respect to $t$ over $E$ and use \eqref{L-1-1a} 
to obtain a contradiction. 

Next, we let $w$ be the solution of
\begin{equation}\label{L-1-4}
\left\{
\aligned
-\Delta w+ \sigma_\e^{-2} V (x/\e) w & =0 & \quad & \text{ in } Q(0, t),\\
w & =u & \quad & \text{ on } \partial Q(0, t),
 \endaligned
 \right.
 \end{equation}
 where $V(x)$ is given by  \eqref{V}.
 It follows by Theorems \ref{cor2-2} and  \ref{versionD2-2} that 
 \begin{equation}\label{L-1-5}
 \|\nabla (u- \chi_{\e, \eta} w)\|_{L^2(Q(0, t))}
 \le C \psi(\eta)
 \left\{ \sigma_\e^{-1} \| w \|_{L^p (Q(0, t))}
 + \| \nabla w \|_{L^p(Q (0, t ))} \right\},
 \end{equation}
 where  $p=\frac{2d}{d-1}$ and $\chi_{\e, \eta}$ is the corrector for the domain $Q_{\e, \eta}(0, t)$.
 This, together with  Theorem \ref{extshro-3}, yields
 \begin{equation}\label{L-1-6}
 \aligned
  \|\nabla (u- \chi_{\e, \eta} w)\|_{L^2(Q(0, t)) }
 & \le C \psi(\eta) 
\left\{ \sigma_\e^{-1} \| u \|_{L^2(\partial Q(0, t))}
+ \|\nabla u \|_{L^2(\partial Q(0, t))} \right\}\\
& \le C \psi(\eta) 
\| \nabla u \|_{L^2(Q(0, 3))},
\endaligned
\end{equation}
where we have used \eqref{L-1-3} for the last step.

Finally, we let $v= \chi_{\e , \eta} w$.
Note that the estimate   \eqref{L-1-1} is given by  \eqref{L-1-6}.
To see \eqref{L-1-2}, we note that
$$
\aligned
\left(\fint_{Q(x, 2\e) } |\nabla v|^2 \right)^{1/2}
 & \le \max_{Q(x, 2\e) } |w|
\left(\fint_{Q(x, 2\e) } |\nabla \chi_{\e, \eta}|^2 \right)^{1/2}
+ \max_{Q(x, 2\e) } |\chi_{\e, \eta} | |\nabla w| \\
& \le C \max_{Q(x, 2\e) } ( \sigma_\e^{-1} |w | +|\nabla w|),
\endaligned
$$
where we have used Lemmas \ref{LPGRAD-2} and \ref{LPGRADD2-2}.
This, together with Remark \ref{re-S1}, gives
\begin{equation*}
\aligned
\left(\fint_{Q(x, 2\e)} |\nabla v|^2 \right)^{1/2}
 & \le C \max_{Q(0, 3/2)} ( \sigma_\e^{-1} |w| + |\nabla w|)\\
& \le C \left\{
\sigma_\e^{-1} \| w \|_{L^2(Q(0, 2))}
+ \|\nabla w \|_{L^2(Q(0, 2))} \right\}\\
 & \le C \|\nabla u \|_{L^2(Q(0, 3))}
\endaligned
\end{equation*}
for any $x\in Q(0, 1)$.
\end{proof}

\begin{lemma}\label{lemma-L2}
Let $u \in H^1(B(x_0, 4R))$,  where $R\ge 8 \e$ and $B(x_0, 8R)\subset \Omega$.
Suppose that $\Delta u  =0$ in $ Q(x_0,4R)\cap \Omega_{\e, \eta}$ and $u=0 $ in $ Q(x_0, 4R)\setminus \Omega_{\e, \eta} $.
Let
\begin{equation}\label{L-2-0}
v (x) =\left(\fint_{Q(x, 2\e) } |\nabla u|^2 \right)^{1/2}.
\end{equation}
Then for $2<p< \infty$, 
\begin{equation}\label{L-2-1}
\left(\fint_{Q(x_0, R)} |v|^p \right)^{1/p}
\le C \left(\fint_{Q(x_0, 2R)} |v|^2 \right)^{1/2},
\end{equation}
where $C$ depends only on $d$, $p$ and $\{Y_z^s\}$.
\end{lemma}

\begin{proof}
Without loss of generality we may assume $x_0=0$.
By dilation we may assume  $R=1$.
To prove \eqref{L-2-1}, we apply Theorem \ref{RV-L} with $F=v$.
Let $Q=Q(y_0, r)$, where $y_0\in Q(0, 1)$ and $0< r< (1/8)$.
If $0< r< 8 \e$, we let $R_Q= v$ and $F_Q=0$.
Note that 
$$
\max_{2Q} R_Q  \le C \left(\fint_{Q(y_0, 2r+2\e)} |\nabla u|^2 \right)^{1/2}
\le C \left(\fint_{4Q} |v|^2\right)^{1/2}.
$$
If $r\ge 8 \e$, we let $R_Q=|\nabla v|$ and $F_Q= |\nabla (u-v)|$,
where $v$ is given by Lemma \ref{lemma-L1}. 
In view of \eqref{L-1-1} and \eqref{L-1-2},
 we obtain  \eqref{L-R2} and \eqref{L-R3} with $\delta= C \psi (\eta)$.
 It follows by Theorem \ref{RV-L}  that the estimate \eqref{L-2-1} holds if $\eta< \eta_0$, where $\eta_0>0$ depends only on $d$, $p$ and $\{Y_z^s\}$.
\end{proof}

Next, we treat the boundary case where $x_0 \in \partial \Omega$.
Since $\Omega$ is Lipschitz, there exists $r_0>0$ such that $B(x_0,  r_0)\cap \Omega = B(x_0, r_0) \cap D$ and
$B(x_0, r_0) \cap \partial\Omega = B(x_0, r_0) \cap \partial D$, where, after a rotation of the coordinate system,  $D$ is given by 
\begin{equation}\label{D}
D= \left\{
(x^\prime, x_d) \in \R^d: \  x^\prime \in \R^{d-1} \text{ and } 
x_d > \varphi  (x^\prime) \right\},
\end{equation}
for some Lipschitz function   $\varphi: \R^{d-1} \to \R$.

\begin{thm}\label{RV2-3}
    Let $2<p<q$ and $D$ be given by \eqref{D}.
     Let $x_0 \in \partial D  $ and  $F \in L^2(Q(x_0, 2R)\cap D)$.
     Suppose that for each cube $Q=Q(y, r) $ with the property that $0< r< c_0R $ and either $y\in Q(x_0, R ) \cap \partial D$ or 
     $4Q  \subset Q(x_0, 2R)\cap  D $, 
    there exist two measurable functions $F_Q$ and $R_Q$ in $ 2Q \cap D$
    such that 
    \begin{equation}\label{BL-R1}
    |F| \le |R_Q| + |F_Q| \quad \text{ in } 2Q \cap D,
    \end{equation}
    \begin{equation}\label{RQ-3}
        \left( \fint_{2Q\cap D }  |R_Q |^q  \right)^{1/q} \leq N \left(\fint_{4Q\cap D}  |F|^2 \right)^{1/2}, 
\end{equation}
\begin{equation}\label{FQ-3}
        \left( \fint_{2Q\cap D}  |F_Q |^2  \right)^{1/2} \leq  \eta \left(\fint_{4Q\cap D}  |F|^2 \right)^{1/2} , 
    \end{equation}
    where $N > 1$ and $0 < c_0 < 1 $. Then there exists $\eta_0 > 0 $,
     depending only on $d, N,  c_0, ,p, q$ and $ \|\nabla \varphi \|_\infty$, 
     with the property that if $0 \leq \eta < \eta_0$,  then $F \in L^p(Q(x_0, R) \cap D) $ and 
    \begin{equation}\label{RHF-3}
        \left( \fint_{Q(x_0, R) \cap D} |F|^p \right)^{1/p} \leq C  \left( \fint_{Q(x_0, 2R) \cap D} |F|^2 \right)^{1/2},
    \end{equation}
    where $C $ depends at most on d, $N$,  $c_0$, $p$, $q$ and $ \|\nabla \varphi \|_\infty$. 
\end{thm}

The proof of  the above theorem can be found in \cite{MR2353255} or \cite[pp. 79-82]{ shen2018periodic}. 

\begin{lemma}\label{lemma-LB1}
Let $\Omega$ be a bounded $C^1$ domain.
Let $u \in H^1(Q(x_0, 4r)\cap \Omega)$, where  $r\ge 8\e$ and $x_0 \in \partial\Omega$. 
Suppose that $\Delta u=0$ in $Q(x_0, 4r) \cap \Omega_{\e, \eta} $ and
$u =0 $ on $Q(x_0, 4r)\cap \partial   \Omega_{\e, \eta}$.
Then there exists $v\in H^1(Q(x_0, 2r)\cap \Omega)$ such that
\begin{equation}\label{LB-1-1}
\left(\fint_{Q(x_0, r)\cap \Omega} |\nabla (u-v)|^2\right)^{1/2}
\le C \psi (\eta)
\left(\fint_{Q(x_0, 3r)\cap \Omega} |\nabla u |^2 \right)^{1/2},
\end{equation}
\begin{equation}\label{LB-1-2}
\left( \fint_{Q(x_0, r)\cap \Omega} 
\left(\fint_{Q(x, 2\e)\cap \Omega} |\nabla v|^2 \right)^{p/2}\right)^{1/p} 
\le C \left(\fint_{Q(x_0, 3r)\cap \Omega } |\nabla u|^2 \right)^{1/2},
\end{equation}
for any $p>2$, 
where  $\psi (\eta)=\eta^{1/2}$ for $d\ge 3$ and $\psi(\eta) =|\ln \eta|^{-1/2}$ for $d=2$.
The constant $C$ depends only on $d$, $p$,  $\{ Y_z^s\} $ and $\Omega$.
\end{lemma}

\begin{proof}
The proof is similar to that of Lemma \ref{lemma-L1}.
Without loss of generality we may assume $x_0=0$.
By dilation we may also assume $r=1$.
Choose $t\in [2, 3]$ such that 
$( Q(0, t) \cap \Omega)_{\e, \eta} = Q(0, t)\cap \Omega_{\e, \eta}$ and 
\begin{equation}\label{LB-1a}
\sigma_\e^{-2} 
\int_{\Omega \cap \partial Q(0, t)} 
| u|^2 + \int_{\Omega\cap \partial Q(0, t)}  |\nabla u|^2 
\le C \int_{\Omega\cap Q(0, 3)} |\nabla u|^2.
\end{equation}
 Let  $w$ be the solution of 
\begin{equation}\label{LB-1-4}
\left\{
\aligned
-\Delta w+ \sigma_\e^{-2} V (x/\e) w & =0 & \quad & \text{ in } Q(0, t)\cap \Omega,\\
w & =u & \quad & \text{ on } \partial ( Q(0, t)\cap \Omega).
 \endaligned
 \right.
 \end{equation}
 As in the proof of Theorem \ref{lemma-L1}, we have
 \begin{equation}
 \| \nabla (u-\chi_{\e, \eta} w) \|_{L^2(Q(0, t)\cap\Omega )}
 \le C \psi (\eta)
 \|\nabla u \|_{L^2(Q(0, 3)\cap \Omega)},
 \end{equation}
 where $\chi_{\e, \eta}$ denotes the corrector for the domain $Q_{\e, \eta}(0, t)= Q(0, t) \cap \Omega_{\e, \eta}$.
 Let $v= \chi_{\e, \eta} w$. Then
 \begin{equation}
 \left(\fint_{Q(x, 2\e) \cap \Omega} |\nabla v|^2 \right)^{1/2}
 \le C \sigma_\e^{-1} \max_{Q(x, 2\e) \cap \Omega } |w| 
+ C \left(\fint_{Q(x, 2\e) \cap \Omega} |\nabla w|^2\right)^{1/2}.
\end{equation}
Since $w=u=0$ on $Q(0, t ) \cap\partial  \Omega$, it follows from Lemma \ref{lemma-S1} that 
\begin{equation}
\aligned
\sigma_\e ^{-1} \max_{Q(0, 3/2)\cap \Omega} |w|
 & \le C \left(\fint_{Q(0, 2) \cap \Omega}
|\nabla w|^2 \right)^{1/2}, \\
\left(\fint_{Q(0, 3/2)\cap \Omega} |\nabla w|^p \right)^{1/p}
 & \le C \left(\fint_{Q(0, 2)\cap \Omega} |\nabla w|^2 \right)^{1/2}
\endaligned
\end{equation}
for $2< p< \infty$.
As a result, we obtain 
\begin{equation*}
\aligned
\int_{Q(0, 1)\cap \Omega}
\left(\fint_{Q(x, 2\e)\cap \Omega } |\nabla v|^2 \right)^{p/2}
 & \le C \left(\fint_{Q(0, 2) \cap \Omega}
|\nabla w|^2 \right)^{p/2}\\
& \le C \big\{
\|\nabla u \|_{L^2(\partial (Q(0, t)\cap \Omega))}
+ \sigma_\e^{-1} \| u \|_{L^2(\partial Q(0, t)\cap \Omega)} \big\}^p,
\endaligned
\end{equation*}
This, together with \eqref{LB-1a}, gives \eqref{LB-1-2}.
\end{proof}

\begin{proof}[Proof of Theorem \ref{thm-L1}.] 
As we discussed earlier, it suffices to prove the reverse H\"older estimate \eqref{L-2}, 
where $0<R<c_0$ and  either $ Q(x_0, 4R) \subset \Omega$ or $x_0\in \partial\Omega$.
Note that the interior case $Q(x_0, 4R) \subset \Omega$ is given by Lemma \ref{lemma-L2}.
For the boundary case where $x_0\in \partial \Omega$,
we apply Theorem \ref{RV2-3}.
To this end, for each cube $Q=Q(y, r)$ with the property that  $0<r< c_0R$ and either  $4Q \subset  Q(x_0, 2R)\cap \Omega$ or 
$y \in Q(x_0, 2R) \cap \partial \Omega$, we need to construct two functions $F_Q$ and $R_Q$
such that \eqref{BL-R1}-\eqref{FQ-3} hold.
Again, the case $4Q\subset Q(x_0, 2R)\cap \Omega$ is given by Lemma \ref{lemma-L2}.
For the case $y\in \partial\Omega$ and $0< r< 8 \e$, we let 
$F_Q=F$ and $R_Q=0$.
Finally, for $r\ge 8 \e$, we use Lemma \ref{lemma-LB1}.
\end{proof}

\begin{thm}\label{thm-4}
Let $\Omega$ be a bounded $C^1$ domain and $\Omega_{\e, \eta}$ be given by \eqref{O-e}.
Then,
\begin{equation}\label{Large-e2}
\| S_\e (F,0)\|_{L^p(\Omega)}
\le  C \min (\sigma_\e, 1) \| F \|_{L^p(\Omega)}
\end{equation}
 for any $2< p< \infty$,
 where $C$ depends only on $d$, $p$, $\Omega$, and $\{Y_z^s\}$.
\end{thm}

\begin{proof}
To prove \eqref{Large-e2}, we consider the operator 
$$
T_\e (F) = \left(\min(\sigma_\e, 1)\right)^{-1} S_\e (F, 0)
$$
and apply Theorem \ref{RV1-3}.
Note that $\| T_\e \|_{L^2 \to L^2} \le C$. Moreover, 
 the reverse H\"older condition \eqref{RHT-3} is reduced to the estimate \eqref{L-2}, 
as in the case of $S_\e (0, f)$.
As a result, by Theorem \ref{RV1-3}, we obtain $\| T_\e \|_{L^p\to L^p} \le C$ for any $p>2$.
This gives the estimate  \eqref{Large-e2} for $p>2$.
\end{proof}

\begin{proof}[Proof of Theorem \ref{thm-S0}]
Since
$$
S_\e (F, f) \le S_\e (F, 0) + S_\e (0, f),
$$
Theorem \ref{thm-S0} follows readily from Theorems \ref{thm-L1} and \ref{thm-4}.
\end{proof}



\section{Lower bounds for $A_p(\Omega_{\e, \eta})$ and $B_p(\Omega_{\e, \eta})$ }\label{section-Sh}

In this section we discuss the sharpness of the estimates in Theorems \ref{main-thm-1} and \ref{main-thm-2} for $A_p(\Omega_{\e, \eta})$ as
well as those in Section \ref{section-B} for $B_p(\Omega_{\e, \eta})$.
Our approach is similar to that used in  \cite{shen2023uniform} for an unbounded and  periodically perforated domain,
\begin{equation}\label{space-6}
\omega_{\epsilon,\eta} = \mathbb{R}^d \backslash \bigcup_{z \in \mathbb{Z}^d} \e (z + \eta \overline{Y^s}), 
\end{equation}
where $Y^s$ is a bounded domain with connected $C^1$ boundary such that $B(0, c_0) \subset Y^s \subset B(0, 1/4)$.
Following an  idea of \cite{jing2020unified}, we consider an $\e$-periodic corrector defined by 

\begin{equation}\label{chiper-6}
    - \Delta \psi_{\e,\eta} = \e^{-2} \eta^{d-2} \qquad \text{in } \ \omega_{\e, \eta}
     \qquad \text{and} \qquad \psi_{\e, \eta} = 0 \qquad \text{on} \ \ \mathbb{R}^d \setminus  \omega_{\e,\eta}.
\end{equation}

\begin{lemma}\label{chieetta1-6}
    Let $\psi_{\e,\eta}$ be the $\e$-periodic corrector defined in \eqref{chiper-6}. 
    Then if $d \geq 3 $, 
    \begin{equation}
        \fint_{Q(0, \e)} \psi_{\e, \eta} \approx 1 \qquad \text{and} \qquad 
        \left(\fint_{Q(0, \e)} |\nabla \psi_{\e, \eta} |^2 \right)^{1/2} \approx \e^{-1} \eta^{\frac{d-2}{2}}. 
    \end{equation}
    If $d = 2$, we have  
     \begin{equation}
        \fint_{Q(0, \e) } \psi_{\e, \eta} \approx | \ln \eta| \qquad \text{and} 
        \qquad \left(\fint_{Q(0, \e) } |\nabla \psi_{\e, \eta} |^2 \right)^{1/2} \approx \e^{-1} | \ln \eta|^{1/2}. 
    \end{equation}
\end{lemma}

\begin{proof}
The case $\e=1$ was proved in  \cite[Lemma 4.4]{shen2023uniform}. The case $0< \e<1$ follows by rescaling.
\end{proof}

\begin{lemma}\label{chieetta2-6}
     Let $\psi_{\e,\eta}$ be the same as in Lemma \ref{chieetta1-6}. Suppose $2< p< \infty$. 
     Then if $d \geq 3 $, 
    \begin{equation}
        \left(\fint_{Q(0, \e)} |\psi_{\e, \eta}|^p \right)^{1/p} \leq C
        \qquad \text{and} \qquad \left(\fint_{Q(0, \e) } |\nabla \psi_{\e, \eta} |^p \right)^{1/p} 
        \geq C\e^{-1} \eta^{\frac{d}{p}-1}. 
    \end{equation}
    If $d = 2$, we have  
  \begin{equation}
        \left(\fint_{Q(0, \e) } |\psi_{\e, \eta}|^p \right)^{1/p} 
        \leq C | \ln \eta| \qquad \text{and}
         \qquad \left(\fint_{Q(0, \e) } |\nabla \psi_{\e, \eta} |^p \right)^{1/p} \geq C \e^{-1} \eta^{\frac{2}{p}-1}. 
    \end{equation}
\end{lemma}

\begin{proof}
The case $\e=1$ was proved in  \cite[Lemma 5.3]{shen2023uniform}, while
the general case follows by rescaling.
\end{proof}

For a bounded Lipschitz domain $\Omega$, define
\begin{equation}\label{O-9e}
\Omega_{\e, \eta} = \Omega \setminus \bigcup_z \e (z + \eta \overline{Y^s}), 
\end{equation}
where  the union is taken over those $z$'s in $\mathbb{Z}^d$  for which $\e Q(z, 1)\subset \Omega$; i.e., $\Omega_{\e, \eta}$ is given by 
\eqref{O-e} with  $x_z=0$ and $Y_z^s =Y^s$ for all $z\in \Z^d$.
Recall that $A_p (\Omega_{\e, \eta})$ and $B_p(\Omega_{\e, \eta})$ denote the smallest bounding constants for which \eqref{W1p-0} holds,
where $u_\e$ is the solution of  the Dirichlet problem \eqref{DP-0} in $\Omega_{\e, \eta}$.
Similarly, we use $C_p (\Omega_{\e, \eta})$ and $D_p(\Omega_{\e, \eta})$ to denote the smallest bounding constants for which the $L^p$ estimate,
\begin{equation}\label{9-0}
\| u_\e \|_{L^p(\Omega_{\e, \eta})}
\le C_p (\Omega_{\e, \eta}) \| f \|_{L^p(\Omega_{\e, \eta})}
+ D_p (\Omega_{\e, \eta}) \|F \|_{L^p(\Omega_{\e, \eta})},
\end{equation}
holds. By a duality argument we have
\begin{equation}\label{dual}
A_p(\Omega_{\e, \eta})
=A_{p^\prime} (\Omega_{\e, \eta}), \quad D_p (\Omega_{\e, \eta} ) =D_{p^\prime}(\Omega_{\e, \eta}) 
\quad \text{ and }
\quad
B_p(\Omega_{\e, \eta}) = C_{p^\prime} (\Omega_{\e, \eta}),
\end{equation}
where $1< p< \infty$ and $p^\prime =\frac{p}{p-1}$.

We begin with the sharpness of $D_p(\Omega_{\e, \eta})$.

\begin{thm}\label{thm-9D}
Let $\Omega$ be a bounded Lipschitz domain and
$\Omega_{\e, \eta}$ be given by \eqref{O-9e}.
 Then, for $d \geq 2$ and $1< p< \infty$,
    \begin{equation}
    D_p(\Omega_{\e, \eta}) \approx \min (\sigma_\e^2, 1).
    \end{equation}
   \end{thm} 

\begin{proof}
The upper bound $D_p(\Omega_{\e, \eta})  \le C \min (\sigma_\e^2, 1)$ is contained in  Theorem \ref{lemma-P1}.
To show the lower bound $D_p(\Omega_{\e, \eta}) \ge c \min (\sigma_\e^2, 1)$,  by \eqref{dual}, we may  assume $1 < p \leq 2$. 
Without loss of generality we may also assume $Q(0, 4r_0) \subset \Omega$.

    Suppose $\e$ is sufficiently small and $R\e \le r_0$ for some $R\in \mathbb{N}$.
    Let $\phi\in C_0^\infty (Q(0, 2R\e ))$ be a cut-off function such that $0 \leq \phi \leq 1$,  $\phi=1$ in $Q(0, R\e)$,
    $|\nabla \phi |\le C/( R\e )$ and $|\nabla^2 \phi |\le C / (R\e)^2$.
    Note that
        \begin{equation}\label{laplace-6}
        - \Delta (\psi_{\e, \eta} \phi) = \e^{-2} \eta^{d-2} \phi - 2 \nabla \psi_{\e, \eta} \cdot \nabla \phi - \psi_{\e,\eta} \Delta \phi
    \end{equation}
   in $\Omega_{\e, \eta} $ and  $\psi_{\e, \eta} \phi =0$ on $\partial \Omega_{\e, \eta}$.
   It follows that
   \begin{equation*}
   \begin{aligned}
&        cR^{d/p} \| \psi_{\e,\eta} \|_{L^p(Q(0, \e) )} 
\leq \| \psi_{\e, \eta} \phi\|_{L^p(\Omega_{\e, \eta}) }\\ 
&\leq C D_p(\Omega_{\e, \eta}) 
\left\{ \e^{-2 + \frac{d}{p}} \eta^{d-2} R^{\frac{d}{p}} 
+ \e^{-1} R^{\frac{d}{p}-1} \| \nabla \psi_{\e, \eta} \|_{L^p(Q(0, \e))} 
+ \e^{-2} R^{\frac{d}{p}-2} \| \psi_{\e, \eta} \|_{L^p(Q(0, \e) )} \right\}, 
        \end{aligned}
   \end{equation*}
where we have used the periodicity of $\psi_{\epsilon, \eta}$.
 This gives 
\begin{equation}\label{Dpres-6}
\begin{aligned}
    &  \e^{\frac{d}{p}} \fint_{Q(0, \e) } |\psi_{\e, \eta}| \\\ 
    &\le C D_p{(\Omega_{\e, \eta} )} \left\{ \e^{\frac{d}{p}-2} \eta^{d-2} 
    + R^{-1} \e^{\frac{d}{p}-1}\left( \fint_{Q(0, \e) } |\nabla \psi_{\e, \eta} |^2 \right)^{1/2} + (\e R)^{-2} \| \psi_{\e, \eta} \|_{L^p(Q(0, \e) )} \right\},
    \end{aligned}
\end{equation}
where we have used the fact $1<p\le 2$.

\medskip

\textit{\textbf{Case 1:} $d \geq 3$ and $0< \sigma_\e \leq 1.  $ }

\medskip

By applying Lemmas \ref{chieetta1-6} and \ref{chieetta2-6} to \eqref{Dpres-6}, we obtain 
\begin{equation}\label{Dd3post-6}
       c  \leq  D_p{(\Omega_{\e, \eta} )} \left\{ \e^{-2} \eta^{d-2} + R^{-1} \e^{-2} \eta^\frac{d-2}{2} + (\e R)^{-2}  \right\}.
\end{equation}
Since $\sigma_\e = \e \eta^{1-\frac{d}{2}}\le 1$,
 we may choose $R \in \mathbb{N}$ such that $R \approx  \eta^{1-\frac{d}{2}}$ in \eqref{Dd3post-6} to get
\begin{equation*}
       c   \leq  D_p{(\Omega_{\e, \eta} )} \cdot  \e^{-2} \eta^{d-2}.
\end{equation*}
Hence, 
$
         D_p(\Omega_{\e, \eta} )\ge c \sigma_\e^2 = c \min (\sigma_\e^2, 1).
$

\medskip

\textit{\textbf{Case 2:} $d \geq 3$ and $\sigma_\e \ge 1$.}

\medskip

In this case we choose $R\in \mathbb{N}$ such that $R \approx \e^{-1}$ in \eqref{Dd3post-6}. It follows that 
\begin{equation*}
    c \leq D_p(\Omega_{\e, \eta}) \left\{ \e^{-2}\eta^{d-2} + \e^{-1} \eta^{\frac{d-2}{2}} + 1 \right\}.
\end{equation*}
Since $\sigma_\e\ge 1$, we obtain 
$
    D_p(\Omega_{\e, \eta} ) \ge c = c \min (\sigma_\e^2, 1).
$

\medskip

\textit{\textbf{Case 3:} $d = 2$ and $\sigma_\e  \leq 1$. } 

\medskip

Applying Lemmas \ref{chieetta1-6} and \ref{chieetta2-6} to \eqref{Dpres-6} now yields
\begin{equation}\label{Dd2post-6}
c | \ln \eta| \leq D_p(\Omega_{\e, \eta} ) \left\{\e^{-2} + R^{-1} \e^{-2} | \ln \eta|^{\frac{1}{2}} + (\e R)^{-2} | \ln \eta| \right\}. 
\end{equation}
Picking $R \in \mathbb{N}$ such that   $R \approx  | \ln \eta|^{\frac{1}{2}}$ in \eqref{Dd2post-6} gives 
\begin{equation*}
    D_p(\Omega_{\e, \eta} )  \ge c\,  \e^2 |\ln \eta| =c \min (\sigma_\e^2, 1).
\end{equation*}

\textit{\textbf{Case 4:} $d = 2$ and $\sigma_\e  \geq 1 $. } 

\medskip

Choosing $R\in \mathbb{N}$ such that $ R \approx  \e^{-1}$ in \eqref{Dd2post-6} yields 
\begin{equation*}
\begin{aligned}
    c | \ln \eta| &\leq D_p(\Omega_{\e, \eta} ) \left\{ \e^{-2} + \e^{-1} |\ln \eta|^{\frac{1}{2}} + | \ln \eta|\right\} \\ 
    &\leq C  D_p(\Omega_{\e, \eta} ) | \ln \eta|. 
\end{aligned}
\end{equation*}
It follows that $ D_p(\Omega_\epsilon)\ge c =c \min (\sigma_\e^2, 1)$.
\end{proof}   

Next, we consider the case $B_p(\Omega_{\e, \eta})$ for $1< p\le 2$.

\begin{thm}\label{CPthm-6}
Let $\Omega$ be a bounded Lipschitz domain and $\Omega_{\e, \eta}$ be given by \eqref{O-9e}.
Then
      \begin{equation}\label{low-B3}
    B_p (\Omega_{\e, \eta}) \approx \min (\sigma_\e, 1)
    \end{equation}
    for $d\ge 2$ and $1< p\le 2$,
\end{thm}

\begin{proof}
It follows from Theorem \ref{lemma-P1} that 
\begin{equation*}
C_p(\Omega_{\e, \eta} ) \le C \min (\sigma_\e, 1)
\end{equation*}
for $2\le p< \infty$.
Since $B_p(\Omega_{\e, \eta}) = C_{p^\prime} (\Omega_{\e, \eta})$, we obtain the upper bound
\begin{equation*}
B_p(\Omega_{\e, \eta}) \le C \min (\sigma_\e, 1)
\end{equation*}
for $1< p\le 2$.
To establish the lower bounds, we first note that
$D_2 (\Omega_{\e, \eta}) \le C\min (\sigma_\e, 1)  B_2 (\Omega_{\e, \eta})$,  which follows \eqref{P00}.
As a result, by Theorem \ref{thm-9D}, 
$B_2 (\Omega_{\e, \eta} ) \ge c \min (\sigma_\e, 1)$.

For the case $1<p< 2$, 
we proceed as in the proof of Theorem \ref{thm-9D}.
It follows from \eqref{laplace-6} that 
    \begin{equation*}
    \begin{aligned}
        & \|\psi_{\e, \eta} \phi \|_{L^q(\Omega_{\e, \eta})}\le C \| \nabla(\psi_{\e,\eta} \phi) \|_{L^p(\Omega_{\e, \eta})}  \\
        & \le C B_p(\Omega_{\e, \eta} ) \left\{ \e^{\frac{d}{p}-2}\eta^{d-2} R^{\frac{d}{p}} 
        + \e^{-1} R^{\frac{d}{p}-1 } \| \nabla \psi_{\e, \eta} \|_{L^p(Q(0, \e) )} + \e^{-2} R^{\frac{d}{p}-2}\| \psi_{\e, \eta} \|_{L^p(Q(0, \e) )}     \right\}, 
        \end{aligned}
    \end{equation*}
    where $\frac{1}{q}=\frac{1}{p} -\frac{1}{d}$.
This implies that 
    \begin{equation}\label{Cppre-6}
         \fint_{Q(0, \e) } |\psi_{\e, \eta}| \leq C B_p(\Omega_{\e, \eta} ) \left\{\e^{-1}\eta^{d-2} R  
         + \left(\fint_{Q(0, \e) } |\nabla \psi_{\e, \eta}|^2\right)^{{1}/{2}} +  (\e R)^{-1}
         \left(\fint_{Q(0, \e)} |\psi_{\e, \eta} |^p \right)^{1/p}  \right\}.
    \end{equation}
    
    \medskip
    
    \textit{\textbf{Case 1:} $d\geq 3$ and $\sigma_\e  \leq 1.  $ } 
    
    \medskip
    
    We apply Lemmas \ref{chieetta1-6} and \ref{chieetta2-6} to \eqref{Cppre-6} to obtain 
    \begin{equation}\label{Cppost1-6}
        c \leq B_p(\Omega_{\e, \eta} ) \left\{\e^{-1} \eta^{d-2} R + \e^{-1} \eta^{\frac{d-2}{2}} + \e^{-1} R^{-1}  \right\}.
    \end{equation}
Letting  $R\approx  \eta^{-\frac{d-2}{2}}$ in \eqref{Cppost1-6} yields 
\begin{equation*}
c \leq B_p(\Omega_\epsilon) \cdot \e^{-1} \eta^{\frac{d-2}{2}}.
\end{equation*}
Thus,
$ B_p(\Omega_{\e, \eta})  \ge  c \sigma_\e = c \min (\sigma_\e, 1)$.

\medskip

\textit{\textbf{Case 2:} $d \geq 3$ and $\sigma_\e  \geq 1. $}

\medskip

Let $R\approx \e^{-1}$ in \eqref{Cppost1-6}. Then 
\begin{equation*}
c \leq B_p(\Omega_{\e, \eta} ) \left\{ \e^{-2} + \e^{-1} \eta^{\frac{d-2}{2}} + C  \right\}.
\end{equation*}
Thus, $B_p(\Omega_{\e, \eta}) \ge c = c \min (\sigma_\e, 1)$.

\medskip

\textit{\textbf{Case 3:} $d=2$ and $ \sigma_\e  \leq 1 $.}

\medskip
 
Applying Lemmas \ref{chieetta1-6} and \ref{chieetta2-6} to \eqref{Cppre-6} now yields
\begin{equation}\label{Cppost2-6}
c | \ln \eta| \leq B_p(\Omega_{\e, \eta} ) \left\{\e^{-1}R +  \e^{-1} | \ln \eta|^{\frac{1}{2}} + (\e R)^{-1} | \ln \eta| \right\}.
\end{equation}
Picking $R \approx   | \ln \eta|^{\frac{1}{2}}$ in \eqref{Cppost2-6} gives 
\begin{equation*}
    c |\ln \eta| \leq B_p(\Omega_{\e, \eta} ) \cdot  \e^{-1} | \ln \eta|^{\frac{1}{2}}.
\end{equation*}
Thus,
$
B_p(\Omega_{\e, \eta}) \ge c \e |\ln \eta|^{1/2} =c \min (\sigma_\e, 1)$.

\medskip

\textit{\textbf{Case 4:} $d=2$ and $\sigma_\e  \geq 1$. }

\medskip

Taking  $R \approx  \e^{-1}$ in \eqref{Cppost2-6} leads to 
\begin{equation*}
    \begin{aligned}
        c | \ln \eta| & \leq B_p(\Omega_{\e, \eta} ) 
        \left\{ \e^{-2} + \e^{-1}| \ln \eta|^{\frac{1}{2}} +  | \ln \eta| \right\} \\ & \leq  C B_p(\Omega_{\e, \eta} )| \ln \eta|. 
    \end{aligned}
\end{equation*}
This implies  that $B_p(\Omega_{\e, \eta})  \ge c =c \min (\sigma_\e, 1)$.
\end{proof}

We now consider $B_p(\Omega_{\e, \eta})$ for $2< p< \infty$.

\begin{thm}\label{B-low-1}
    Let $\Omega_{\e, \eta}$ be given by \eqref{O-9e}.
     Assume $d \geq 3$. Then if $0< \sigma_\e \leq 1$ and $2< p< \infty$, 
    \begin{equation}\label{Low-B-1}
B_p (\Omega_{\e, \eta} ) 
\approx  \e \eta^{1-d + \frac{d}{p}}.
\end{equation}
If  $\sigma_\e \geq 1$ and $2< p< d$, we have 
\begin{equation}\label{low-B-2}
B_p(\Omega_{\e, \eta} ) \approx 
1+ \e^{-1} \eta^{\frac{d}{p}-1}.
\end{equation}
\end{thm}

\begin{proof} 
Suppose $0< \sigma_\e\le 1$. The upper bound 
$B_p(\Omega_{\e, \eta}) \le C \e \eta^{1-d +\frac{d}{p}}$ is contained in Theorem \ref{main-thm-3}.
To prove the lower bound, we use the Poincar\'e inequality, 
\begin{equation}\label{P-p0}
\| u \| _{L^p(\Omega_{\e, \eta})}
\le C \min (\e \eta^{1-\frac{d}{p}}, 1) \|\nabla u \|_{L^p(\Omega_{\e, \eta})}
\end{equation}
for $u \in W^{1, p}_0(\Omega_{\e, \eta})$, where $2< p< d$. 
The inequality \eqref{P-p0} follows readily from Lemma \ref{lemma-P0}.
As a result, we see that 
\begin{equation}\label{Low-B-3}
D_p(\Omega_{\e, \eta}) \le C B_p (\Omega_{\e, \eta}) \min (  \e \eta^{1-\frac{d}{p}}, 1) .
\end{equation}
By Theorem \ref{thm-9D}, we have $D_p(\Omega_{\e, \eta}) \ge c \sigma_\e^2= c \e^2 \eta^{2-d}$.
This, together with \eqref{Low-B-3}, gives \eqref{Low-B-1} for
the case $2< p< d$.

Next, we use a convexity argument to handle  the case $0< \sigma_\e \le 1$ and $d\le p< \infty$.  
Choose $2< q< d$ and $t \in (0, 1)$ so that
$$
\frac{1}{q} =\frac{1-t}{2} +\frac{t}{p}.
$$
 By Riesz-Thorin Theorem we have
 $$
 B_q (\Omega_{\e, \eta} )\le \left[ B_2 (\Omega_{\e, \eta})\right]^{1-t} \left[ B_p (\Omega_{\e, \eta})\right]^t.
 $$
 It follows that 
 $$
 B_p(\Omega_{\e, \eta})
 \ge \left[ B_q(\Omega_{\e, \eta}) \right]^{\frac{1}{t} } \left[ B_2 (\Omega_{\e, \eta})\right]^{1-\frac{1}{t}}.
 $$
 The desired estimate follows by using $ B_2(\Omega_{\e, \eta})  \le C \sigma_\e$ and $B_q(\Omega_{\e, \eta})
 \ge c \e \eta^{1-d +\frac{d}{q}}$.
 
 Finally,  we consider the case $\sigma_\e\ge 1$ and $2< p< d$.
 The upper bound is contained in Theorem \ref{main-thm-4}. 
 Since $D_p(\Omega_{\e, \eta})\ge c$, 
 by \eqref{Low-B-3}, we obtain 
    \begin{equation*}
        B_p(\Omega_{\e, \eta} ) \geq 
c \e^{-1} \eta^{\frac{d}{p}-1}
    \end{equation*}
However, \eqref{Low-B-3} also yields $B_p (\Omega_{\e, \eta}) \ge c D_p(\Omega_{\e, \eta}) \ge c$.
As a result, we obtain $B_p(\Omega_{\e, \eta}) \ge c(1+ \e^{-1} \eta^{\frac{d}{p}-1})$.
\end{proof}

Finally, we establish lower bounds for $A_p(\Omega_{\e, \eta})$.

\begin{thm}\label{thm-A-low-1}
Let $ 2 <  p <  \infty$.
Let $\Omega_{\e, \eta}$ be given by \eqref{O-9e}.
    Then, if  $0< \sigma_\e \leq 1 $, 
\begin{equation}\label{LA-1}
    A_p(\Omega_{\e, \eta}) \geq 
    \begin{cases}
    \begin{alignedat}{2}
   & c \eta^{-d|\frac{1}{2}-\frac{1}{p}|} \qquad &&\text{if} \ d \geq 3, \\ 
   & c \eta^{-2|\frac{1}{2}-\frac{1}{p}|}|\ln \eta|^{-\frac{1}{2}} \qquad && \text{if} \ d=2.
    \end{alignedat}
     \end{cases}
\end{equation}
If $\sigma_\epsilon \geq 1 $ we have 
\begin{equation}\label{LA-2}
A_p(\Omega_{\e, \eta} ) \geq  \begin{cases}
    \begin{alignedat}{2}
        & c(1+ \e^{-1}\eta^{\frac{d}{p}-1}) \qquad && \text{if} \  d\ge 3 \text{ and  } 2 < p <d, \\
        &c\e^{-1}\eta^{\frac{d}{p}-1} \qquad && \text{if} \ d\ge 3 \text{ and } d \le  p < \infty, \\ 
        &c \e^{-1}\eta^{\frac{2}{p} - 1 } |\ln \eta|^{-1} \qquad && \text{if} \ d = 2   \text{ and } 2< p< \infty.
    \end{alignedat}
    \end{cases}
\end{equation}
 The constants $c>0$ only depend on $d$, $p$,  $Y^s$ and $\Omega$.
\end{thm}

\begin{proof}
    Let $u_\e$ be a solution to $-\Delta u_\e = \text{\rm div} (f)$ in $\Omega_{\e, \eta}$ 
    with $u_\e = 0$ on $\partial \Omega_{\e, \eta}$. 
    By a Sobolev imbedding, we have 
    \begin{equation*}
        \| u_\e \|_{L^{q'}(\Omega_{\e, \eta} )} \leq C \| \nabla u_\e\|_{L^{p'}(\Omega_{\e, \eta})} 
        \leq C A_p(\Omega_{\e, \eta} ) \| f\| _{L^{p'}(\Omega_{\e, \eta} )}, 
    \end{equation*}
where $\frac{1}{q'} = \frac{1}{p'} - \frac{1}{d}$ and $1 < p' <d.$ 
By duality this implies that if $-\Delta v_\e = G$ in $\Omega_{\e, \eta} $ and $v_\e = 0$ on $\partial \Omega_{\e, \eta},$ then
\begin{equation*}
    \| \nabla v_\e \| _{L^p(\Omega_{\e, \eta})} \leq C A_p(\Omega_{\e, \eta}) \| G\|_{L^q(\Omega_{\e, \eta} )}.
\end{equation*}
Thus,  if $-\Delta u_\e = F + \text{\rm div} (f)$ in $\Omega_{\e, \eta} $ with $u_\e =0$ on $\partial \Omega_{\e, \eta}$, then 
\begin{equation}\label{soblevAp-6}
    \| \nabla u_\e  \|_{L^p(\Omega_{\e, \eta} )} \leq C A_p(\Omega_{\e, \eta} ) 
    \left\{ \| F\|_{L^q(\Omega_{\e, \eta} )} + \| f\| _{L^p(\Omega_{\e, \eta})} \right\}, 
\end{equation}
where $d' < p < \infty$ and $\frac{1}{q} = \frac{1}{p} + \frac{1}{d}$. 

Let $\phi$ and $\psi_{\e, \eta}$ be the same as in the proof of Theorem \ref{thm-9D}. Note that  
\begin{equation*}
    -\Delta(\psi_{\e,\eta} \phi) = \e^{-2} \eta^{d-2} \phi - 2 \text{\rm div} (\psi_{\e,\eta} \nabla \phi) + \psi_{\e, \eta} \Delta \phi 
\end{equation*}
in $\Omega_{\e, \eta} $  and $\psi_{\e,\eta} \phi = 0 $ on $\partial \Omega_{\e, \eta}$. 
It follows  from \eqref{soblevAp-6} that 
\begin{equation*}
\begin{aligned}
    &R^{\frac{d}{p}} \| \nabla \psi_{\e,\eta}\|_{L^p(Q(0, \e))}
    \le  C  \| \nabla ( \psi_{\e,\eta} \phi)\|_{L^p(\Omega_{\e, \eta})} \\
    &\leq 
    C A_p(\Omega_{\e, \eta} ) 
    \left\{\e^{\frac{d}{q}-2} \eta^{d-2}R^{\frac{d}{q}} 
    + \e^{-1} R^{\frac{d}{p} - 1 } \| \psi_{\e, \eta} \|_{L^p(Q(0, \e))} 
    + \e^{-2} R^{\frac{d}{q}-2} \| \psi_{\e,\eta} \|_{L^q(Q(0, \e) )}  \right\},
    \end{aligned}
\end{equation*}
 where we have used the periodicity of $\psi_{\e,\eta}$. Hence,
\begin{equation}\label{Appre-6}
\aligned
&    \| \nabla \psi_{\e,\eta}\|_{L^p(Q(0, \e))} \\
&    \leq  C A_p(\Omega_{\e, \eta} ) 
    \left\{ \e^{\frac{d}{p}-1} \eta^{d-2}R + \e^{-1} R^{- 1} \| \psi_{\e, \eta} \|_{L^p(Q(0, \e) )} + \e^{-2} R^{-1} \| \psi_{\e,\eta} \|_{L^q(Q(0, \e) )} \right\}.
\endaligned
\end{equation}



\textit{\textbf{Case 1:} $d \geq 3$ and $0< \sigma_\e  \leq 1.$} 
 
 \medskip

Applying Lemma \ref{chieetta2-6} to \eqref{Appre-6} yields 
\begin{equation}\label{Apost1-6}
   \eta^{\frac{d}{p} -1 } \leq C  A_p(\Omega_{\e, \eta} ) \left\{ \eta^{d-2}R +  R^{- 1}  \right\}.
\end{equation}
Since $\sigma_\e\le 1$, we may pick 
 $R \approx  \eta^{- \frac{d-2}{2}}$ in \eqref{Apost1-6} and obtain 
\begin{equation*}
    A_p(\Omega_{\e, \eta}) \geq C \eta^{\frac{d}{p} - \frac{d}{2}} = C \eta^{-d|\frac{1}{p}-\frac{1}{2}|}.
\end{equation*}

\textit{\textbf{Case 2:} $d =2$ and $ 0< \sigma_\e  \leq 1.$}

\medskip

Applying Lemma \ref{chieetta2-6} to \eqref{Appre-6} now yields 
\begin{equation}\label{apost2-6}
    \eta^{\frac{2}{p}-1} \leq C A_p(\Omega_{\e, \eta}) \left\{   R + | \ln \eta|R^{-1}   \right\}.
\end{equation}
Picking $R \approx  | \ln \eta|^{\frac{1}{2}}$ in \eqref{apost2-6} gives 
\begin{equation*}
     \eta^{\frac{2}{p}-1} \leq C  A_p(\Omega_{\e, \eta})   | \ln \eta|^{\frac{1}{2}},
\end{equation*}
which leads to 
\begin{equation*}
A_p(\Omega_{\e, \eta} ) \geq c \eta^{-2|\frac{1}{2}- \frac{1}{p}|} | \ln \eta|^{-\frac{1}{2}}.
\end{equation*}

\textit{\textbf{Case 3:} $d \geq 3$ and $\sigma_\e \geq 1.$}

\medskip

Choosing $R \approx  \e^{-1}$ in \eqref{Apost1-6} yields 
\begin{equation*}
\begin{aligned}
    \eta^{\frac{d}{p}-1} &\leq C A_p(\Omega_{\e, \eta} ) \left\{ \eta^{d-2} \e^{-1} + \e \right\} \\
    & \leq C \e A_p(\Omega_{\e, \eta} ).
    \end{aligned}
\end{equation*}
Thus, 
\begin{equation*}
    A_p(\Omega_{\e, \eta} ) \geq c\,  \e^{-1}\eta^{\frac{d}{p}-1}.
\end{equation*}
This gives the lower bound in \eqref{LA-2} for the case $d\le p< \infty$.
Note that  if  $2<p< d$, by the Poincar\'e inequality $ \| u_\e \|_{L^p(\Omega_{\e, \eta})}
\le C \|\nabla u_\e \|_{L^p(\Omega_{\e, \eta})}$, we have 
\begin{equation*}
    A_p(\Omega_{\e, \eta} ) \geq  c\,  C_p(\Omega_{\e, \eta} ) = c\,  B_{p'}(\Omega_{\e, \eta} ) \geq c. 
\end{equation*}
Hence, 
$$
A_p(\Omega_{\e, \eta} ) \ge c ( 1+ \e^{-1} \eta^{\frac{d}{p}-1}).
$$

\textit{\textbf{Case 4:} $d=2$ and $\sigma_\e \geq 1.$}

\medskip

Now we  take $R \approx \e^{-1}$ in \eqref{apost2-6}. This gives 

\begin{equation*}
    \begin{aligned}
        \eta^{\frac{2}{p}-1} & \leq C A_p(\Omega_{\e, \eta} )
         \left\{ \e^{-1} + \e | \ln \eta| \right\} \\ &\leq C \e A_p(\Omega_{\e, \eta} ) | \ln \eta|,
    \end{aligned}
\end{equation*}
which leads to 
\begin{equation*}
    A_p(\Omega_{\e, \eta} ) \geq C \e^{-1} \eta^{\frac{2}{p} - 1 } | \ln \eta|^{-1}.
\end{equation*}
This completes the proof.
\end{proof}

\bibliographystyle{plain}

\bibliography{Righi-Shen.bbl}

\medskip

\begin{flushleft}

Robert Righi,
Department of Mathematics,
University of Kentucky,
Lexington, Kentucky 40506,
USA.

E-mail: robert.righi@uky.edu

\medskip

Zhongwei Shen,
Department of Mathematics,
University of Kentucky,
Lexington, Kentucky 40506,
USA.

E-mail: zshen2@uky.edu
\end{flushleft}

\bigskip

\medskip

\end{document}